\documentclass{amsart}

\usepackage{amssymb,amscd}
\usepackage{amsmath,latexsym,amssymb, amscd}
\usepackage{amssymb,latexsym, xspace, enumerate}

\usepackage{graphicx}

\newtheorem{theorem}{Theorem}[section]
\newtheorem{lemma}[theorem]{Lemma}
\newtheorem{proposition}[theorem]{Proposition}
\newtheorem{corollary}[theorem]{Corollary}

\long\def\alert#1{\smallskip{\hskip\parindent\vrule%
\vbox{\advance\hsize-2\parindent\hrule\smallskip\parindent.4\parindent%
\narrower\noindent#1\smallskip\hrule}\vrule\hfill}\smallskip}

\theoremstyle{definition}
\newtheorem{definition}[theorem]{Definition}
\newtheorem{example}[theorem]{Example}
\newtheorem{examples}[theorem]{Examples}

\newtheorem{discussion}[theorem]{Discussion}
\newtheorem{remark}[theorem]{Remark}
\newtheorem{question}[theorem]{Question}

\numberwithin{equation}{section}

\newtheorem{remark/questions}[theorem]{Remark and Questions}
\newtheorem{fact}[theorem]{Fact}
\newtheorem{remarks}[theorem]{Remarks}

 \long\def\alert#1{\smallskip\line{\hskip\parindent\vrule
\vbox{\advance\hsize-2\parindent\hrule\smallskip\parindent.4\parindent
  \narrower\noindent#1\smallskip\hrule}\vrule\hfill}\smallskip}

\def\frac#1#2{{{#1}\over{#2}}}

\def\endresult#1{\medskip}

\def \Q{{{\mathbb Q }}}
  \def \N{{\mathbb  N}}

\def\q{{\bf q}}
\def\p{{\bf p}}
\def\m{{\bf m}}
\def\n{{\bf n}}

\def\dim{\mathop{\rm dim}}

\def\Spec{\mathop{\rm Spec}}
\def\spec{\mathop{\rm Spec}}

\def\hgt{\mathop{\rm ht}}

\begin{document}

\baselineskip 15 pt

\title[Examples of non-Noetherian domains]{
Examples of non-Noetherian domains \\ inside power series rings }
\author{William Heinzer}
\address{Department of Mathematics, Purdue University, West
Lafayette, Indiana 47907}

\email{heinzer@math.purdue.edu}

\author{Christel Rotthaus}
\address{Department of Mathematics, Michigan State University, East
  Lansing,
   MI 48824-1027}

\email{rotthaus@math.msu.edu}

\author{Sylvia Wiegand}
\address{Department of Mathematics, University of Nebraska, Lincoln,
NE 68588-0130}

\email{swiegand@math.unl.edu}

\subjclass{ Primary 13B35, 13J10, 13A15}



\keywords{power series, Noetherian and non-Noetherian integral domains.}

   \thanks{The authors are grateful for
  the hospitality, cooperation and support of Michigan State,
  Nebraska,
   Purdue and the CIRM at Luminy, France, where several work sessions on this research were
   conducted.}

   \begin{abstract}
Given a power series ring $R^*$ over a Noetherian
integral domain $R$ and an intermediate field $L$ between $R$ and
the total quotient ring of $R^*$, the integral domain $A = L \cap
R^*$ often (but not always) inherits nice properties from $R^*$
such as the Noetherian property. For certain fields $L$ it is possible to
approximate $A$ using a localization $B$ of a particular nested
union of polynomial rings over $R$ associated to $A$;  if $B$ is
Noetherian then $B = A$. If $B$ is not Noetherian, we can
sometimes identify   the  prime ideals of $B$ that are not
finitely generated.  We have obtained in this way,  for
each positive integer $m$, a three-dimensional local unique
factorization domain $B$ such that   the
maximal ideal of $B$ is two-generated, $B$ has precisely $m$ prime
ideals of height two, each prime ideal of $B$ of height two is not
finitely generated and all the other prime ideals of $B$ are finitely
generated. We examine the structure of the map $\Spec A \to \Spec
B$ for this example. We also present a generalization of this example
to  dimension four. This four-dimensional non-Noetherian local
unique factorization domain has exactly one prime ideal $Q$ of height
three, and $Q$ is not finitely generated.

\end{abstract} \maketitle

\section{Introduction}\label{intro}

In this paper we analyze the prime
ideal structure of particular non-Noetherian integral domains
arising from a general construction developed in
our earlier papers \cite{noesloc},\dots, \cite{nonfgprimes}. With
this technique two types of integral domains are constructed: 

(1)
The intersection of an ideal-adic completion $R^*$ of a Noetherian integral domain
$R$  with an appropriate  subfield of the total quotient 
ring of $R^*$  yields an integral domain $A$ as in the abstract, and 

(2) An approximation
of the domain $A$  by a nested union $B$ of localized polynomial
rings has the second form described in the abstract.

Recently there has been considerable interest in non-Noetherian
analogues of Noetherian notions such as the concept of ``regular"
ring. Rotthaus and Sega have shown that the rings $A$ and $B$
produced in the general construction are {\it coherent regular}
local rings in the sense that every finitely generated
submodule of a
free module has a finite free resolution;
see \cite{RS} and  Remark~\ref{coherence}.

 We construct in this paper rings that are not Noetherian but are very close to being Noetherian,
in that localizations at most prime ideals are Noetherian and most prime ideals are
finitely generated; sometimes  just one prime ideal is not finitely generated. If a ring
has exactly one prime ideal that is not finitely generated, that prime ideal contains all
nonfinitely generated ideals of
the ring. 

This article expands upon  previous work of the authors  where we 
construct non-Noetherian local domains of dimension $d \ge 3$ 
with some of these properties  \cite{insideps} and \cite{nonfgprimes}.
In the case of  dimensions three
and four, we give 
considerably more detail in this article 
 about  these non-Noetherian domains.  
In particular we  categorize the height-one primes of the three-dimensional example in terms 
of  the spectral map from $A$ to $B$.

In Section~\ref{2} we describe examples of three-dimensional non-Noetherian 
non-catenary unique factorization
domains.   Another  example of a three-dimensional non-Noetherian 
unique factorization domain is given by John David in \cite{David}.  
The examples given in Examples~\ref{16.5.1} 
are very close to being Noetherian.  
We give more details about a specific case
where there is precisely one nonfinitely generated prime ideal in Example~\ref{16.5.3de}.
In Section~\ref{3} we give
background results that apply in a more general setting. Our main results are in Sections~\ref{4}
and \ref{5}. Section~\ref{4} contains the verification of the properties of the three-dimensional
examples. 

In Example~\ref{16.5.4de} of Section~\ref{5},  we construct 
a four-dimensional non-Noetherian non-catenary local
unique factorization domain $B$ that again is close to being Noetherian.
The ring $B$  has exactly one prime ideal $Q$ of height
three, and $Q$ is not finitely generated. We leave open the question of 
whether there exist any prime ideals of $B$ of height two that are not
finitely generated. Following a suggestion of the referee,  we use a ``$D+M$'' 
construction to obtain in Example~\ref{referee} a four-dimensional non-Noetherian non-catenary local
domain $C$;  the maximal ideal of $C$  is principal and is the only nonzero 
finitely generated prime ideal of $C$.

All rings we consider are assumed to be commutative with identity. A general
reference for our notation and terminology is \cite{M}. We abbreviate unique factorization domain
 to  UFD, regular local domain to RLR  and  discrete rank one valuation domain to DVR.

\section{A family of examples in dimension 3} \label{2}

In this section we construct examples as described in Examples~\ref{16.5.1}. In
the next section
we give a diagram and more detail for a special case of the example with exactly one nonfinitely generated prime ideal.

\begin{examples}\label{16.5.1}
For each positive integer $m$, we construct an example of a non-Noetherian
 local integral  domain $(B, \n)$
such that:
\begin{enumerate}
\item $\dim B = 3$.
\item The ring $B$ is a UFD that is not catenary.
\item The maximal ideal  $\n$ of $B$ is 2-generated.
\item The $\n$-adic completion of $B$ is a two-dimensional regular local domain.
 \item For every non-maximal prime ideal $P$ of $B$,  the ring $B_P$ is Noetherian.
\item The ring $B$ has precisely $m$ prime ideals of height two.
\item Each prime ideal of $B$ of height two is not
finitely generated; all other prime ideals of $B$ are finitely generated.
\end{enumerate}
\end{examples}

To establish the existence of the examples in Example~\ref{16.5.1},
we use the following notation.
Let $k$ be a field, let $x$ and $y$  be indeterminates
over $k$, and set
$$
R : ~~=~~  k[x, y]_{(x,y)},   \quad K~:= ~k(x,y) \quad\text{and} \quad R^*:~~=~~k[y]_{(y)}[[x]].
$$
The power series ring  $R^*$ is  the $xR$-adic completion of $R$.  Let $\tau\in
xk[[x]]$ be  transcendental over $k(x)$.
For each integer $i$ with $1 \le i \le m$, let $p_i \in R \setminus  xR$
 be such
that $p_1R^*, \ldots, p_mR^*$ are $m$  prime ideals.
For  example, if each $p_i\in R \setminus (x,y)^2R$, then each  $p_iR^*$ is prime in $R^*$.
In particular one could take $p_i = y - x^i$.
Let $p := p_1 \cdots p_m$. We set  $f := p\tau$ and consider
the injective $R$-algebra homomorphism $S := R[f] \hookrightarrow
R[\tau] =: T$.  In this construction the polynomial rings $S$ and $T$ have
the same field of fractions $K(f) = K(\tau)$. The intersection domain
\begin{equation*}
A ~:= ~ R^* ~\cap ~ K(f)  ~ = ~ R^* ~\cap ~ K(\tau)
\tag{\ref{16.5.1}.0}
\end{equation*}
 is a two-dimensional regular local domain with maximal
ideal $(x,y)A$ and the $(x,y)A$-adic completion of $A$ is $k[[x, y]]$,   \cite{V}.

Let $\tau := c_1x + c_2x^2 + \cdots + c_ix^i + \cdots \in xk[[x]]$, where
the $c_i\in k$ and, for each non-negative integer $n$, define the
``$n^{\text{th}}$ endpiece'' $\tau_n$ of $\tau$  by
\begin{equation}
\tau_n ~:= ~\sum_{i=n+1}^\infty c_ix^{i-n} ~ = ~\frac{\tau - \sum_{i=1}^nc_ix^i}{x^n}.
\tag{(\ref{16.5.1}.a)}
\end{equation}
We have the following relation between $\tau_n$ and $\tau_{n+1}$ for each $n$:
\begin{equation}
\tau_n ~ = ~ c_{n+1}x ~+~ x\tau_{n+1}. \tag{\ref{16.5.1}.b}
\end{equation}
Define $f_n := p\tau_n$, set  $U_n= R[f_n]=k[x,y]_{(x,y)}[f_n]$, a 3-dimensional
polynomial ring  over $R$, and set $B_n=(U_n)_{(x,y,f_n)}=k[x,y,f_n]_{(x,y,f_n)}$,
a 3-dimensional localized
polynomial ring. Similarly set $U_ {\tau n}=R[\tau_n]=k[x,y]_{(x,y)}[\tau_n]$, a
3-dimensional
polynomial ring containing $U_n$,  and  $B_{\tau n}=k[x,y,\tau_n]_{(x,y,\tau_n)}$,
a localized polynomial ring  containing $U_{\tau n}$ and $B_n$.
Let $U, B, U_\tau$ and $B_\tau$ be the nested union domains  defined as follows:
$$
U ~ :=  ~ \bigcup_{n=0}^\infty U_n  ~ \subseteq ~ U_\tau~ :=  ~ \bigcup_{n=0}^\infty U_{\tau n};
 ~ \qquad B ~ :=  ~ \bigcup_{n=0}^\infty B_n  ~ \subseteq ~ B_\tau~ :=
 ~ \bigcup_{n=0}^\infty B_{\tau n}  ~ \subseteq ~ A.
$$
\begin{remark} \label{16.5.11}
By Equation \ref{16.5.1}.b, $k[x,y,f_{n+1}]\subseteq k[x,y,f_n][1/x]$ for each $n$.
Thus $k[x, y, f][1/x] ~= ~ k[x, y, f_n][1/x] $
and
\begin{equation*}
~ ~U[1/x]=R[f][1/x] ;\qquad ~ ~\qquad ~ ~U_\tau[1/x]=R[\tau][1/x] . \tag{\ref{16.5.11}.0}
\end{equation*}
Hence, for each $n$, the ring $B_n[1/x]$ is a localization of $S =U_0= R[f]$. It
follows that $B[1/x]$ is a localization of $S$ and $B[1/x]$ is a localization of $B_n$.
Similarly, $B_\tau[1/x]$ is a
localization of $T = R[\tau]$.\end{remark}

We establish in Theorem~\ref{16.5.2} of Section~\ref{4} that the rings $B$ of
Examples~\ref{16.5.1} have
 properties 1 through 7 and also some additonal properties.

Assuming  properties 1 through 7 of Examples~\ref{16.5.1}, we describe the ring $B$
of Examples~\ref{16.5.1}
in the case where $m = 1$ and $p = p_1 = y$.

\begin{example} \label{16.5.3de} Let the notation be as in Examples~\ref{16.5.1}.
Thus
$$
R=k[x,y]_{(x,y)}, \quad f=y\tau, \quad f_n=y\tau_n, \quad B_n =
R[y\tau_n]_{(x, y, y\tau_n)}, \quad B = \bigcup_{n=0}^\infty B_n.
$$
  As we show in Section~4, this ring $B$ has exactly one prime ideal
$Q := (y,\{y\tau_n \}_{n
  =0}^\infty)B$   of height 2. Moreover, $Q$ is not finitely
  generated and is the only prime ideal of $B$ that is not finitely
  generated. We also have    $Q =yA\cap B$, and $Q\cap
  B_n=(y,y\tau_n)B_n$ for each $n \ge 0$.

  To identify the ring $B$ up to isomorphism, we include the following details:
By Equation~\ref{16.5.1}.b,  we have  $\tau_n = c_{n+1}x + x\tau_{n+1}$. Thus we have
\begin{equation*}
f_n ~= ~ xf_{n+1} ~+~yxc_{n+1}.
\tag{\ref{16.5.3de}.1}
\end{equation*}
The  family of equations   (\ref{16.5.3de}.1)  uniquely determines $B$ as a nested union
of the 3-dimensional RLRs $B_n = k[x, y, f_n]_{(x,y,f_n)}$.

We recall the following terminology of \cite[page 325]{ZSII}.

\begin{definition}  \label{16.554}   If a ring  $C$ is  a subring of a ring $D$, a prime ideal
$P$ of $C$  is {\it  lost} in $D$ if $PD \cap C \ne  P$.
\end{definition}

\begin{discussion} \label{types}
Assuming  properties 1 through 7 of Examples~\ref{16.5.1},
if $q$ is a height-one prime of $B$, then $B/q$ is Noetherian if
  and only if $q$ is not contained in $Q$. This is clear since  $q$ is principal and $Q$
  is the unique prime of $B$ that is not finitely generated,  and a
  ring is Noetherian if each prime ideal of the ring is finitely
  generated; see \cite[Theorem~3.4]{M1}.

The height-one primes $q$ of  $B$ may be  separated into several
  types as follows:

\noindent
{\bf Type  I.}
The primes $q \not \subseteq Q$  have the
property that  $B/q$ is a one-dimensional Noetherian local domain.
These primes are
  contracted from $A$, i.e., they are not lost in $A$. To see this,
consider  $q=gB$ where $g
  \not\in Q$. Then $gA$ is contained in a height one prime $P$ of
  $A$. Hence  $g \in (P \cap B) \setminus Q$ so $P \cap B \ne Q$.
  Since $\m_B A = \m_A$, we have $P \cap B \ne \m_B$. Therefore $P
  \cap B$ is a height-one prime containing $q$, so $q = P \cap B$
  and $B_q = A_P$.

  There are infinitely many primes $q$ of type I, because  every
  element of $\m_B \setminus Q$ is contained in a prime $q$ of type
  I. Thus $\m_B \subseteq Q \cup \bigcup\{ q \text{ of Type   I} \}$.
  Since $\m_B$ is not the union of finitely many strictly smaller
  prime ideals, there are infinitely many primes $q$ of Type   I.

\smallskip

\noindent
{\bf Type   I*.} Among the primes of Type   I,  we label the prime ideal $xB$ as Type I*.
The prime ideal $xB$ is special since it is the  unique  height-one prime $q$
of $B$ for which $R^*/qR^*$ is not complete. If $q$ is a height-one prime of $B$
such that $x \notin qR^*$, then $x \notin q$ by Proposition~\ref{16.5.15}.4.
Thus $R^*/qR^*$ is complete with
respect to the powers of the nonzero principal ideal generated by the image of $x$
mod $qR^*$.
Notice that  $R^*/xR^* \cong k[y]_{yk[y]}$.

\smallskip
If $q$ is a height-one prime of $B$ not of Type   I, then $\overline
  B= B/q$ has precisely three prime ideals. These prime ideals form
  a chain: $(\overline 0)\subset \overline Q\subset \overline
  {(x,y)B} = \overline{\m_B}$.

\noindent
{\bf Type   II.} We define the primes of Type   II to be the
  primes  $q \subset Q$  such that  $q$ has height  one and is
  contracted from a prime  $p$ of $A=k(x,y, f)\cap R^*$, i.e., $q$ is
not lost in $A$. For example,
  the prime   $y(y+\tau)B$ is of Type   II, by Lemma~\ref{lost}.  For
$q$ of this type, $B/q$ is dominated
  by the one-dimensional Noetherian local domain $A/p$. Thus $B/q$
  is a non-Noetherian {\it generalized local ring} in the sense of Cohen; that is, 
  $B/q$ has a unique maximal ideal $\overline\n$ that is finitely generated and $\cap_{i=1}^\infty{\overline{\n}}\,^i=(0)$,  \cite{Co}.

  For $q$ of Type   II,   the maximal ideal of $B/q$ is not principal.
  This follows because a  generalized local domain having
  a principal maximal ideal is a DVR \cite[(31.5)]{N2}.

 There are infinitely many height-one primes of Type   II, for
  example, $y(y + x^t\tau)B$ for each $t \in \N$; see Lemma~\ref{prime}.
For $q$ of Type   II,
  the DVR $B_q$ is birationally dominated by $A_p$. Hence $B_q =
  A_p$ and the ideal $\sqrt{qA} = p \cap yA$.

  That each element $y(y + x^t\tau)$ is irreducible, and thus generates a
height-one prime ideal,
  is  done in greater generality in Lemma~\ref{prime}.

\medskip
\noindent
{\bf Type   III.}
  The primes of Type   III are the primes  $q \subset Q$ such that  $q$ has
  height one and is  not contracted from $A$, i.e., $q$ is lost in $A$.
For example, the prime
  $yB$ and the prime $(y+x^ty\tau)B$ for $t \in \N$ are of Type   III; see  Lemma~\ref{lost}.
Since the
  elements $y$ and $y + x^ty\tau$ are in $\m_B$ and are not in
  $\m_B^2$ and since $B$ is a UFD, these elements are necessarily
  prime. There are infinitely many such prime ideals by Lemma~\ref{prime}.
 For $q$ of Type   III, we have $\sqrt{qA} = yA$.

  If $q = yB$ or $q = (y + x^ty\tau)B$, then the image
  $\overline{\m_B}$ of $\m_B$ in $B/q$ is principal. It follows that
  the intersection of the powers of $\overline{\m_B}$ is $Q/q$, and so
  $B/q$ is not a generalized local ring. To see that 
$\bigcap_{i=1}^\infty (\overline{\m_B})^i    \ne(0) $, we argue as follows: If $P$ is a principal
  prime ideal of a ring and $P'$ is a prime ideal properly contained
  in $P$, then $P'$ is contained in the intersection of the powers
  of $P$; see   \cite[page 7, ex. 5]{Kap}.
\end{discussion}
\end{example}

The picture of  $\Spec(B)$ is
shown below.
\vskip 10 pt

\setbox4=\vbox{\hbox{%
\hskip 70pt
     \rlap{\kern  1.050in\lower 0.166in\hbox to .3in{\hss $\m_B:=(x,y)B$\hss}}%
        \rlap{\kern  2.450in\lower .666in\hbox to .3in{\hss $Q:=(y,\{f_i\})B$\hss}}%
   \rlap{\kern  .050in\lower 1.266in\hbox to .3in{\hss $xB\in\boxed{\text{Type   I}}$\hss}}%
   \rlap{\kern  1.350in\lower 1.266in\hbox to .3in{\hss $\boxed{\text{Type   II}}$\hss}}%
   \rlap{\kern  2.850in\lower 1.266in\hbox to .3in{\hss $yB\in\boxed{\text{Type   III}}$\hss}}%
\rlap{\kern  1.350in\lower 1.666in\hbox to .3in{\hss $(0)$\hss}}%
     \rlap{\special{pa 1100 180}    \special{pa 150 1100}    \special{fp}}%
\rlap{\special{pa 1100 180}    \special{pa 2500 550}    \special{fp}}%
    \rlap{\special{pa 2500 750}    \special{pa 1500 1100}    \special{fp}}%
    \rlap{\special{pa  2500 750}    \special{pa 2800 1100}    \special{fp}}%
    \rlap{\special{pa 350 1360}    \special{pa 1350 1570}    \special{fp}}%
    \rlap{\special{pa 1450 1360}    \special{pa 1350 1570}    \special{fp}}%
    \rlap{\special{pa 2850 1360}    \special{pa 1350 1570}    \special{fp}}%
   }} \box4 \vskip 10 pt

\centerline{Diagram~\ref{16.5.3de}.2}

In Remarks~\ref{17.45}  we examine the height-one primes of $B$ from a different perspective.

\begin{remarks} \label{17.45} (1) Assume the notation of Example~\ref{16.5.3de}.
If $w$ is a nonzero prime element of $B$ such that $w \notin Q$, then $wA$ is a prime
ideal in $A$ and is the unique prime ideal of $A$ lying over $wB$.
To see this,  observe that $w \notin yA$ since $w \notin Q = yA \cap B$.
It follows that if  $p \in \Spec A$ is a minimal
prime of $wA$, then  $y \notin p$. Thus  $p \cap B \ne Q$, and so, since we assume the properties
of Examples~\ref{16.5.1} hold,  $p \cap B$ has height one.
Therefore $p \cap B = wB$. Hence the DVR $B_{wB}$ is birationally dominated by $A_p$,
and thus $B_{wB} = A_p$. This implies that  $p$ is the unique prime of $A$ lying over $wB$. We also have
$wB_{wB} = pA_p$. Since $A$ is a UFD and $p$ is the unique minimal prime of $wA$,
it follows that $wA = p$.
In particular, $q$ is not lost in $A$; see Definition~\ref{16.554}.

If $q$ is a
height-one prime of $B$ that is contained in $Q$, then $yA$ is a
minimal prime of $q$, and $q$ is of
Type   II or III depending on whether or not $qA$ has other  minimal prime divisors.

To see this, observe that if  $yA$ is the only prime divisor of $qA$, then $qA$ has
radical $yA$ and $yA \cap B = Q$ implies that $Q$ is the radical of $qA \cap B$.
Thus $q$ is
lost in $A$  and  $q$ is of  Type   III.

On the
other hand, if
there is a minimal prime $p \in \Spec A$ of $qA$ that  is different
from $yA$, then $y$ is not in $p \cap B$  and hence $p \cap B \ne Q$. Since
$Q$ is
the only prime of $B$ of height two, it follows that
$p \cap B$  is a height-one prime and thus $p \cap B  = q$. Thus
$q$ is not lost in $A$ and  $q$ is of Type   II.

We observe that for every  Type   II
prime $q$ there are exactly two minimal primes of $qA$, one of these
is $yA$ and the other  is a height-one prime  $p$ of $A$
such that $p \cap B = q$. For every height-one prime ideal $p$ of $A$ such
that $p \cap B = q$, we have  $B_q$ is a
DVR that is birationally dominated by $A_p$ and hence $B_q = A_p$.   The
uniqueness of $B_q$ implies that there is precisely one such prime ideal $p$ of $A$.

An example of a height-one prime ideal $q$ of Type   II is
$q :=  (y^2+y\tau)B$.  Then  $qA =
(y^2+y\tau)A$ has the two minimal primes $yA$ and $(y+\tau)A$.

(2)  The ring  $B/yB$ is a rank $2$ valuation ring. This can be seen
directly or else one may apply \cite[Prop. 3.5(iv)]{HS}.   For other
prime elements $g$ of $B$ with $g \in Q$, it need not be true that
$B/gB$ is a valuation ring.  If $g$ is a prime element contained in
$\m_B^2$,  then the maximal ideal of $B/gB$ is 2-generated but not
principal and thus $B/gB$ cannot be a valuation ring.   For a
specific example over the field $\Q$,  let   $g = x^2+y^2\tau$.
   \end{remarks}

\section{Background results} \label{3}

We use results  from a general
construction developed in our earlier papers. In particular, we
use the following theorem in establishing Examples~\ref{16.5.1}.

\begin{theorem} \label{build} \text{ \cite[Theorem~1.1]{padua97},  \cite[Theorem~3.2]{noehom},
\cite{powerbook}.}~
Let $R$ be a Noetherian
integral domain with field of fractions $K$. Let $a$ be a nonzero
nonunit of $R$ and let $R^*$ be the $(a)$-adic completion of $R$.
Let $h$ be a positive integer and let
$\tau_1,\dots, \tau_h\in aR[[a]]=aR^*$, abbreviated by ${\underline\tau}$,  be
algebraically
independent over $K$. Let $U_{\underline\tau}$ and $C_{\underline\tau}$ be defined as follows
$$
U_{\underline\tau}~:= ~\bigcup_{r=0}^\infty U_{\underline \tau r}  \quad \text{and}
\quad C_{\underline \tau} ~:=~
\bigcup_{r=0}^\infty C_{\underline \tau r},
$$
where for each integer $r \ge 0$, $U_{\underline \tau r} := R[\tau_{1r},
\ldots, \tau_{hr}]$, ~
$C_{\underline \tau r} := (1 + aU_{\underline \tau r})^{-1}U_{\underline \tau r}$, and each
$\tau_{ir}$ is the $r^{\text{th}}$ endpiece of $\tau_i$ defined as in Equation~\ref{16.5.1}.a.
Then the following statements
are equivalent:
\begin{enumerate}
\item
$A_{\underline\tau}:=K(\underline\tau)\cap R^*$ is Noetherian and
$A_{\underline\tau}=C_{\underline\tau}$.
\item
$A_{\underline\tau}$ is Noetherian and is a localization of a subring of $U_{\underline \tau 0}[1/a]$.
\item
$A_{\underline\tau}$ is Noetherian and is a localization of a subring of $U_{\underline \tau}[1/a]$.

\item $U_{\underline\tau}$ is Noetherian.
\item $C_{\underline\tau}$ is Noetherian.
\item  $R[{\underline\tau}]\to  R^*[1/a]$ is
flat.
\item  $C_{\underline\tau} \to  R^*[1/a]$ is
flat.
\end{enumerate}
\end{theorem}

Propositions~\ref{16.5.15} and ~\ref{Bufd} are used for Examples~\ref{16.5.1}.

\begin{proposition} \label{16.5.15} With the notation of Theorem~\ref{build},
let $U = U_{\underline \tau}, ~ U_n = U_{\underline \tau n}$,  $C = C_{\underline \tau}$
and $A = A_{\underline \tau}$.
Then for all $t \in \N$ we have
\begin{equation*}
a^tC ~ = ~ a^tR^* \cap C \quad \text{ and }
\quad \frac{R}{a^tR} ~ =  ~ \frac{U}{a^tU} ~=  ~ \frac{C}{a^tC} ~=
~ \frac{A}{a^tA} ~=
~ \frac{R^*}{a^tR^*}. \tag{\ref{16.5.15}.0}
\end{equation*}
Moreover,
\begin{enumerate}
\item  The $(a)$-adic completions of $U$, $C$ and $A$
are all equal to  $R^*$, and $a$ is in the Jacobson radical of $C$.
\item The ring $U[1/a] = R[\underline \tau][1/a]$, and so $C[1/a]$
 is a localization of $R[\underline \tau]$.
\item If $q$ is  a prime ideal of $R$, then $qU$ is a prime ideal of $U$,
and either $qC = C$ or $qC$ is a prime ideal of $C$.
\item
Let $I$ be an ideal of  $C$ and let $t \in \N$. If   $a^t \in IR^*$,
then  $a^t \in I$.
\item \label{znpna} Let $P\in\spec C$  with $a\notin P$. Then
$a$ is a nonzerodivisor on $R^*/PR^*$. Thus $a \notin Q$ 
for each associated prime $Q$ of the ideal $PR^*$. Since $a$ is in
the Jacobson radical of $R^*$, it follows that $PR^*$ is contained 
in a nonmaximal prime ideal of $R^*$. 
\item  \label{Bloc} If $R$ is local, then $R^*$ and  $C                                                                                                                                                                                                                                                                                                                                                                                                                                                                                                                                                                                                                                                                                                                                                                                                                                                                                                                                                                                                 $ are both local,
 and we let $\m_R, ~\m_{R^*}$ and $\m_{C}$ denote
the maximal ideals of $R, ~R^*$ and $C$, respectively. In this case
\begin{itemize}
\item  $\m_C = \m_RC$  and each prime ideal $P$ of $C$ such that $\hgt(\m_C/P) = 1$ is contracted
from $R^*$.
\item If an ideal $I$ of $C$ is such that  $IR^*$ is primary for $\m_{R^*}$,
then $I$ is primary for $\m_C$.
\end{itemize}
\end{enumerate}
\end{proposition}

\begin{proof} The  equalities  in Equation~\ref{16.5.15}.0 follow
from \cite[Prop. 2.2.1]{padua97}, \cite[Prop. 2.4.3]{noehom},
\cite[Corollary~6.19]{powerbook}, and these imply item 1 about $(a)$-adic completions. For the second statement,
since $C_n = (1+aU_N)^{-1}U_n$,
it follows that $1 + ac$ is a unit of $C_n$ for each $c \in C_n$. Therefore
$a$ is in the Jacobson radical  of $C_n$ for each $n$ and thus $a$ is
in the Jacobson radical of $C$.

For item~2, the
relation given in Equation~\ref{16.5.1}.b for the case of one variable $\tau$ holds
also in the case of several variables and implies that $U[1/a] = R[\underline \tau][1/a]$.
Since $C$ is a localization of $U$, we have
$C[1/a]$ is a
localization of $R[\underline \tau]$ by Remark~\ref{16.5.11}.

For item~3, since each $U_n$ is a polynomial ring over $R$, $qU_n$ is a prime ideal of $U_n$
and thus $qU = \bigcup_{n=0}^\infty qU_n$ is a prime ideal of $U$. Since $C$ is a localization
of $U$, $qC$ is either $C$ or a prime ideal of $C$.

To see item~4, observe that
there  exist elements $b_1, \ldots, b_s \in I$ such
that $IR^* = (b_1, \ldots,b_s)R^*$. If $a^t \in IR^*$, there
exist $\alpha_i \in R^*$ such that
$$
a^t ~= ~ \alpha_1b_1 + \cdots + \alpha_sb_s.
$$
We have $\alpha_i = a_i + a^{t+1}\lambda_i$ for each $i$, where $a_i \in C$ and
$\lambda_i \in R^*$. Thus
$$
a^t[1 - a(b_1\lambda_1 + \cdots + b_s\lambda_s)] ~= ~ a_1b_1 + \cdots + a_sb_s
\in C\cap a^tR^*=a^tC.
$$
Therefore  $\gamma := 1 - a(b_1\lambda_1 + \cdots +
b_s\lambda_s)
\in C$.  Thus $a(b_1\lambda_1 + \cdots +
b_s\lambda_s)\in C\cap aR^*=aC$, and so  $b_1\lambda_1 + \cdots +
b_s\lambda_s\in C$.   By item~1, the element  $a$ is
in the Jacobson radical of $C$. Hence
$\gamma$ is invertible in $C$. Since $\gamma a^t\in (b_1,\cdots, b_s)C$,
it follows that $a^t \in I$.

For item~\ref{znpna},  assume that $P\in\spec C$ and $a\notin P$. We have that
$$P\cap aC=aP\quad\text{ and so }\quad \frac P{aP}=\frac{P}{P\cap aC}
\cong \frac{P+aC}{aC}$$
By  Equation~\ref{16.5.15}.0,  $C/aC$ is Noetherian.  Hence  the  $C$-module 
$C/aC$ is finitely generated. Let $g_1,\dots, g_t\in P$ be such that 
$P=(g_1,\dots, g_t)C+aP$. Then also $PR^*=(g_1,\dots, g_t)R^*+aR^*=(g_1,\dots, g_t)R^*$; 
the first equality is by Equation~\ref{16.5.15}.0, and the last equality is by Nakayama's Lemma. 

Let $\widehat f\in R^*$ be such that $a\widehat f\in PR^*$.  we show that 
$\widehat f\in PR^*$.

Since $\widehat f\in R^*$, we 
have  $\widehat f:=\sum_{i=0}^\infty c_ia^i$, 
where each $c_i\in R$. For each $m>1$, 
let $f_m:=\sum_{i=0}^{m}c_ia^i$, the first $m+1$ terms of 
this expansion of $\widehat f$. 
Then  $f_m\in R\subseteq C$ and there exists an element $\widehat h_1\in R^*$ so that. 
$$\widehat f=f_m+a^{m+1}\widehat {h_1}.$$
Since $a\widehat f\in PR^*$, we have  
 $$a\widehat f=\widehat {a_1}g_1+\dots + \widehat {a_t}g_t, $$
 for some $\widehat{a_i}\in R^*$. The $\widehat {a_i}$ have
 power series expansions in $a$ over $R$, and thus there exist elements $a_{im} \in R$ 
such that $\widehat{a_i} - a_{im} \in a^{m+1}R^*$.  Thus
  $$a\widehat f=a_{1m}g_1+\dots + a_{tm}g_t + a^{m+1}\widehat{h_2}, $$
  where  $\widehat h_2\in R^*$, and 
  $$af_m=a_{1m}g_1+\dots + a_{tm}g_t + a^{m+1}\widehat{h_3}, $$
 where  $\widehat{h_3}=\widehat{h_2}-a\widehat{h_1}\in R^*$. 
Since the $g_i$ are in $C$, we have 
 $a^{m+1}\widehat{h_3}\in a^{m+1}R^*\cap C=a^{m+1}C$, the last equality  by
Equation~\ref{16.5.15}.0. Therefore $\widehat{h_3}\in C$.
Rearranging the last set-off equation above, we obtain 
   $$a(f_m- a^{m}\widehat{h_3})=a_{1m}g_1+\dots +  a_{tm}g_t \in P. $$
Since $a\notin P$, we have  
$f_m  - a^{m}\widehat{h_3}\in P$.
It follows that $\widehat f\in P+a^{m}R^*\subseteq PR^*+a^{m}R^*$, 
for each $m>1$.
Hence we have that $\widehat f\in PR^*$, as desired.

For item~\ref{Bloc}, if $R$ is local, then  $C$ is local since $C/aC=R/aR$ and $a$ is in the Jacobson radical of $C$.
Hence also $\m_C = \m_RC$.
If $a \not\in P$,
then item 4 implies that no power of $a$ is in $PR^*$. Hence $PR^*$ is contained in a prime
ideal $Q$ of $R^*$ that does not meet the multiplicatively
closed set  $\{a^n\}_{n=1}^\infty$. Hence $P \subseteq Q\ \cap C \subsetneq \m_C$.
Since $\hgt(\m_C/P) = 1$, we have $P = Q \cap C$, so $P$ is contracted from $R^*$.
If $a \in P$, then
(\ref{16.5.15}.0) implies that $PR^*$ is a prime ideal of $R^*$ and $P = PR^* \cap C$.

For the second part of item~5, if $IR^*$ is $\m$-primary then $a^t\in IR^*$. Thus $a^t\in I$ by item 4.
By Equation~\ref{16.5.15}.0,  $C/a^tC=R^*/a^tR^*$ and so $I/a^tC$ is primary for the maximal ideal of $C/a^tC$. Therefore   $I$ is primary for the maximal ideal of $C$.\end{proof}

The definition of $B$ as a directed union as given in Examples~\ref{16.5.1} and later in
this article is not the same as the definition of $C$
as a directed union given in Theorem~\ref{build} and Propositions~\ref{16.5.15} and \ref{Bufd}.
However the ring $B$ {\it is} the same  as the ring $C$ for $R$ as in Examples~\ref{16.5.1}.
We show this  more generally in Remark~\ref{locBsame}.1 for $R$   a Noetherian local domain.

\begin{remarks} \label{locBsame}
(1) Assume the setting of Theorem~\ref{build} with the additional assumption that
 $R$ is a Noetherian local domain with maximal ideal $\m$.  We observe in this case that
$C_{\underline{\tau} }$ as defined in Theorem~\ref{build} is the directed union of the localized polynomial rings   $B_r := (U_{\underline{\tau} r})_{P_r}$,
where $P_r  := {(\m, \tau_{1r}, \ldots, \tau_{hr})}U_{\underline{\tau} r}$.

\begin{proof} It is clear that  $B_r \subseteq B_{r+1}$,  and  $P_{r}  \cap
(1 + aU_{\underline{\tau} r}) = \emptyset$  implies
that $C_{\underline{\tau} r} \subseteq B_{ r}$.
We show that $B_{r} \subseteq C_{\underline{\tau} r + 1}$: Let $\frac{u}{d} \in B_r$,
where $u \in U_{\underline{\tau} r}$ and $d \in U_{\underline{\tau} r} \setminus P_{ r}$.
Then $d = d_0 + \sum_{i=1}^h \tau_{ir}b_i$,
where $d_0 \in R$ and each $b_i \in U_{\underline{\tau} r}$.
Notice that $d_0 \notin \m$ since $d \notin P_{r}$,
and so $d_0^{-1} \in R$. Thus $dd_0^{-1} = 1 +
\sum_{i=1}^h \tau_{ir}b_id_0^{-1} \in (1 + aU_{\underline{\tau} r+1})$
since each $\tau_{ir} \in aU_{\underline{\tau} r+1}$ by (\ref{16.5.1}.b).
Hence $\frac{u}{d} = \frac{ud_0}{dd_0} \in C_{\underline{\tau} r+1}$,
and so $C_{\underline{\tau}} = \bigcup_{r=1}^\infty C_{\underline{\tau} r}  =
\bigcup_{r=1}^\infty B_r$.
\end{proof}

(2) With the notation of Examples~\ref{16.5.1}, where
$R$ is the  localized polynomial ring $k[x,y]_{(x,y)}$ over a field $k$,
 $R^*=k[y]_{(y)}[[x]]$ is the $(x)$-adic completion of $R$ and $\tau\in xR^*$
is transcendental over $K$,  the proof in
 item (1) shows that
$C_\tau=\bigcup B_r,$  where $B_r=(U_r)_{P_r}$,
$U_r=k[x,y,\tau_r]$ and $P_r  = (x, y,\tau_{r})U_r$.  A
similar remark applies to $C_f$ with appropriate modifications to
$B_r$, $U_r$ and $P_r$.

\smallskip

(3) Thus the results of  Theorem~\ref{build}  and Propositions~\ref{16.5.15} and \ref{Bufd}
  hold for the ring $B$ of Examples~\ref{16.5.1} and also the examples later in this article.

\end{remarks}

\begin{proposition} \label{Bufd}
Assume the notation of Theorem~\ref{build} and set $C:=C_{\underline \tau}$.
\begin{enumerate}
\item
If $R$ is a UFD and $a$ is a prime element of $R$, then $aC$ is a prime ideal,
$C[1/a]$ is a Noetherian UFD and $C$ is a UFD.
\item
If in addition $R$ is regular, then
$C[1/a]$ is a regular Noetherian UFD.
\end{enumerate}
 \end{proposition}

\begin{proof} By Proposition~\ref{16.5.15}.3,  $aC$ is a prime ideal.
Since $R$ is a Noetherian UFD and $S= R[{\tau_1},\cdots, {\tau_h}]$ is a
polynomial ring extension of
$R$, it follows that  $S$ is a Noetherian UFD.  By Remark~\ref{16.5.11},
the ring $C[1/ a]$ is a localization of $S$ and thus a Noetherian UFD;
moreover $C[1/ a]$ is regular
if $R$ is.  The $(a)$-adic completion
of $C$ is $R^*$ by Proposition~\ref{16.5.15}.1. Since $R^*$ is Noetherian
and $a$ is in the Jacobson
radical of $R^*$; see  \cite[Theorem 8.2(i)]{M}, it
follows that $\bigcap_{n=1}^\infty a^nR^* = (0)$. Thus $\bigcap_{n=1}^\infty a^nC
= (0)$ by Equation~\ref{16.5.15}.0. It follows that $C_{aC}$ is Noetherian \cite[(31.5)]{N2}, and hence  $C_{aC}$ is a DVR.
We use the following fact:

\begin{fact} \label{Intdomint} If $D$ is an integral domain and $c$ is a nonzero element
of $D$  such that $cD$ is a prime ideal, then $D=D[1/c]\cap D_{cD}$.
\end{fact}
{\it Proof of fact.}   Let
$\beta \in D[1/c] \cap D_{cD}$. Then $\beta=
\frac{b}{c^n}=\frac{b_1}s$ for some $b, b_1 \in D$, $s\in D\setminus
cD$ and integer $n \ge 0$. If $n> 0$, we have $sb=c^nb_1\implies
b\in cD$. Thus we may reduce to the case where  $n = 0$;  it follows
that $D = D[1/c] \cap D_{cD}$. This proves the fact.

\medskip

We return to the proof of Proposition~\ref{Bufd}. By the fact,   $C = C[1/a] \cap C_{aC}$,
and therefore $C$ is a Krull domain. Since  $C[1/a]$ is a UFD and
$C$ is a Krull domain, it follows
that  $C$ is a UFD \cite[page~21]{Sam}.
 \end{proof}

In order to  examine more closely the prime ideal structure of the ring $B$
of Examples~\ref{16.5.1}, we establish in Proposition~\ref{16.5.17} some
properties of its overring $A$ and of the map $\Spec A \to \Spec B$.

\begin{proposition} \label{16.5.17} With the notation of Examples~\ref{16.5.1}, we have
\begin{enumerate}
\item $A = B_{\tau}$ and $A[1/x]$ is a localization of $R[\tau]$.
\item For $P \in \Spec A$ with $x \notin P$,  the following are equivalent:
$$
{\bf(a)}~ A_P = B_{P \cap B} \hskip 40pt {\bf (b)} ~\tau \in B_{P \cap B} \hskip 40pt
 {\bf(c)} ~p \notin P.
$$
\end{enumerate}
\end{proposition}

\begin{proof}
For item 1, to see that
 $A =  B_{\tau}$,  we first show that the  map
$$
\varphi: R[\tau]\to R^*[1/x]=k[y]_{(y)}[[x]][1/x]
$$ is flat.
By \cite[p. 46]{M}, the field of fractions $L$ of $k[x]_{(x)}[\tau]$  is flat over
$k[x]_{(x)}[\tau]$  since it is  a localization. The field   $k[[x]][1/x]$
contains  $L$ and is flat over $L$ since it has a vector space basis over $L$.
Thus  the map   $\psi: k[x]_{(x)}[\tau]\to k[[x]][1/x]$ is flat.
We use the following:
\begin{fact} \label{tensorflat} Let $C$ be a commutative ring, let $D$, $E$ and $F$ be
$C$-algebras, and let  $\psi: D \to E$ be a flat $C$-algebra homomorphism; equivalently,
$E$ is a flat $D$-module via the $C$-algebra homomorphism $\psi$. Then
$\psi\otimes_C1_F: D \otimes_CF \to E \otimes_CF$
is a flat $C$-algebra homomorphism; equivalently, $E\otimes_CF$ is a
flat $D\otimes_CF$-module via the $C$-algebra homomorphism $\psi\otimes_C 1_F$.
\end{fact}
\begin{proof} Since $E$ is a flat $D$-module, $E \otimes_D(D \otimes_CF)$ is a flat
$(D\otimes_CF)$-module by \cite[p.~46, ~Change of coefficient ring]{M}. The fact follows because
$E \otimes_D(D \otimes_CF) = E\otimes_CF$.
\end{proof}

We return to the proof of Proposition~\ref{16.5.17}. We have  $R = k[x,y]_{(x,y)}$.
Consider the following composition:
$$
R[\tau] = k[x]_{(x)}[\tau] \otimes_{k[x]_{(x)}}R
\overset{\alpha}\to  k[[x]][1/x] \otimes _{k[x]_{(x)}}R   \overset{\gamma}\hookrightarrow
k[[x]][y]_{(x,y)}[1/x].
$$
By Fact~\ref{tensorflat},  the map $\alpha$ is flat. The map $\gamma$ is a localization.
Hence the composition $ \gamma \circ \alpha$ is flat.
The extension $k[[x]][y]_{(x,y)} \to k[y]_{(y)}[[x]]$ is flat since
it is the map taking a Noetherian ring
to an ideal-adic completion \cite[Corollary~1, p. 170]{M1}. Therefore the localization map
$\beta: k[[x]][y]_{(x,y)}[1/x] \to k[y]_{(y)}[[x]][1/x]$ is flat. Thus the map
$\varphi = \beta \circ \gamma \circ \alpha: R[\tau] \to k[y]_{(y)}[[x]][1/x]$ is flat.
Theorem~\ref{build}  implies  that  $A =  B_{\tau}$.
By Remark~\ref{16.5.11}, the ring $A[1/x]$ is a localization of $R[\tau]$.

For item 2,  since $\tau \in A$, (a) $\implies$ (b) is clear. For (b) $\implies$ (c) we show that
$p \in P \implies \tau \notin B_{P \cap B}$.  By Remark~\ref{16.5.11}, $B[1/x]$ is a localization
of $R[f]$.
Since $x \notin P$, the ring
$B_{P \cap B}$ is a localization of $R[f]$, and thus
$B_{P \cap B} = R[f]_{P \cap R[f]}$. The assumption that $p \in P$ implies that some
$p_i \in P$, and so $R[f]_{P \cap R[f]}$ is contained in the DVR  $V:=R[f]_{p_iR[f]}$.
Since $R[f]$ is a
polynomial ring over $R$, $f$ is a unit in $V$.
Hence $\tau = f/p \notin V$ and thus $\tau\notin R[f]_{P\cap R[f]}$.
This shows that (b) $\implies$ (c).

For (c) $\implies$ (a), notice  that $f=p\tau$ implies that $R[f][1/xp]=R[\tau][1/xp]$.
By item 1,
$A[1/x]$ is a localization of $R[\tau][1/x]$ and so $A[1/xp]$ is a localization
of $R[\tau][1/xp]=
R[f][1/xp]$. Thus $A[1/xp]$ is a localization of $R[f]$. By Remark~\ref{16.5.11}, $B[1/x]$
is a localization of $R[f]$. Since $xp\notin P$ and $x\notin P\cap B$, we have that $A_P$
and $B_{P\cap B}$
 are both localizations of $R[f]$. Thus we have
$$ A_P=R[f]_{PA_P\cap R[f]}=R[f]_{(P\cap B)B_{P\cap B}\cap R[f]}=B_{P\cap B}.
$$
This completes the proof of Proposition~\ref{16.5.17}.
\end{proof}

We observe in Proposition~\ref{16.5.18} that over a perfect field $k$ of
characteristic $p > 0$ (so that $k = k^{1/p}$)  a one-dimensional form of the construction
yields a DVR that is not a Nagata ring, and thus not excellent; see \cite[p. 264]{M},
\cite[Theorem~78, Definition~34.8]{M1}.

\begin{proposition} \label{16.5.18} Let $k$ be a perfect field of characteristic $p > 0$
and let $\tau \in xk[[x]]$ be such that $x$ and $\tau$ are algebraically independent
over $k$. Let $V := k(x, \tau) \cap k[[x]]$. Then $V$ is a DVR for which
 the integral closure $\overline{V}$ of $V$ in
the purely inseparable field extension $k(x^{1/p}, \tau^{1/p})$ is not a finitely
generated $V$-module. Hence  $V$ is not a Nagata ring.
\end{proposition}

\begin{proof}
It is clear that $V$ is a DVR with maximal ideal $xV$.
Since $x$ and $\tau$ are algebraically
independent over $k$, $[k(x^{1/p}, \tau^{1/p}): k(x, \tau)] = p^2$.
Let $W$ denote
the integral closure of  $V$  in the field extension $k(x^{1/p}, \tau)$ of
degree $p$ over $k(x, \tau)$. Notice that
$$
W ~ = ~ k(x^{1/p}, \tau) ~\cap ~ k[[x^{1/p}]] \quad \text{and} \quad
\overline{V} ~ = ~ k(x^{1/p}, \tau^{1/p}) ~\cap ~ k[[x^{1/p}]]
$$
are both DVRs having residue field $k$ and maximal ideal generated by $x^{1/p}$.
Thus $\overline{V} = W + x^{1/p}\overline{V}$. If $\overline{V}$ were a finitely
generated $W$-module, then by Nakayama's Lemma it would follow that $W = \overline{V}$.
This is impossible because $\overline{V}$ is not birational over $W$. It follows
that $\overline{V}$ is not a finitely generated $V$-module, and hence $V$ is not a
Nagata ring.
\end{proof}

\section{Verification of the 3-dimensional examples} \label{4}

In Theorem~\ref{16.5.2} we record and establish the properties asserted
in  Examples~\ref{16.5.1} and other properties of the ring $B$.

\begin{theorem} \label{16.5.2} With the notation of Example~\ref{16.5.1}, let
$Q_i:=p_iR^*\cap B$, for each $i$ with $1\le i\le m$. We have:
\begin{enumerate}
\item  The ring $B$ is a three-dimensional non-Noetherian local UFD with
maximal ideal $\n=(x,y)B$, and
the $\n$-adic completion
of $B$ is the two-dimensional regular local ring $k[[x,y]]$.

\item The rings $B[1/x]$ and $B_P$, for each nonmaximal prime ideal $P$ of $B$, are
regular Noetherian UFDs, and
the ring $B/xB$ is a DVR.
\item The ring $A$ is a two-dimensional regular local domain
with maximal ideal $\m_A := (x, y)A$, and $A=B_{\tau}$.
The ring $A$ is excellent if  the field $k$ has characteristic zero.
If $k$ is a perfect field of characteristic $p$, then $A$ is not excellent
\item The ideal
$\m_A$  is the only prime ideal of $A$ lying over
$\n$.
\item The ideals $Q_i$ are the only height-two prime ideals of $B$.
\item The ideals $Q_i$ are not finitely generated and they are the only
nonfinitely generated prime ideals of $B$.

\item  The ring $B$ has saturated chains of prime ideals from $(0)$ to $\n$
of length two and of  length three, and
hence is not catenary.

\end{enumerate}
\end{theorem}

\begin{proof}  For item 1, since $B$ is a directed union of three-dimensional
regular local domains,
 $\dim B \le 3$. By Proposition~\ref{16.5.15},  $B$ is local with maximal ideal $(x,y)B$, $xB$
and $p_iB$ are prime ideals, and the $(x)$-adic completion of $B$
is equal to $R^*$, the $(x)$-adic completion of $R$. Thus  the $\n$-adic completion of $B$ is
$k[[x,y]]$.
 Since each $Q_i=\bigcup_{i=1}^\infty Q_{in}$, where $Q_{in}=p_iR^*\cap B_n$,
 we see that each $Q_i$ is a prime ideal of $B$ with $p_i,f \in Q_i$ and $x\notin  Q_i$.
Since $p_iB=\bigcup p_iB_n$, we have $f\notin p_iB$.
Thus $$(0)\subsetneq p_iB\subsetneq Q_i\subsetneq (x,y)B.$$
 This chain of prime ideals of length at least three yields that $\dim B = 3$ and
that the height of each $Q_i$ is $2$.

The map $ S = R[f] \to R^*[1/x]$ is not flat since flat extensions satisfy  the going-down
property \cite[Theorem~9.5, p. 68]{M}, and
$p_iR^*[1/x]$ is a height-one prime whereas  $p_iR^*[1/x] \cap S = (p_i, f)S$ is a
height-two prime.  Therefore Theorem ~\ref{build} implies that the  ring $B$ is not Noetherian.
  By Proposition~\ref{Bufd}, $B$ is a UFD, and so item 1 holds.

For item 2, by Equation~\ref{16.5.15}.0, $B/xB$ is  a DVR.
By Proposition~\ref{Bufd}, $B[1/x]$ is a
 regular Noetherian UFD.  If $x\in P$ and $P$ is nonmaximal, then, again
by Equation~\ref{16.5.15}.0,  $P=xB$.  If $x \not\in P$, the ring $B_{P}$ is a
localization of $B[1/x]$ and so  is a regular Noetherian UFD. Thus item 2 holds.

 The statement in  item 3 that $A$ is a two-dimensional regular local domain
with maximal ideal $\m_A = (x,y)A$ follows by a result of Valabrega
\cite{V}; see  Equation~\ref{16.5.1}.0. By Proposition~\ref{16.5.17}.1, we have
$A=B_\tau$.
The ring $V := k[[x]] \cap k(x, \tau)$ is a DVR by \cite[(33.7)]{N2}.
If  the field $k$ has characteristic zero, then
$V$ is excellent by  \cite[Chap IV]{G}, \cite[Folgerung~3]{R0}. Since
$A$ is a
localization of $V[y]$, it follows that $A$ is also excellent if
$k$ has characteristic zero.

Assume the field $k$ is perfect with characteristic $p > 0$. By
Proposition~\ref{16.5.18}, the ring $V$ is not excellent. Since
$A = V[y]_{(x,y)}$, the ring $V$ is a homomorphic image of $A$. Since excellence is
preserved under homomorphic image, the ring $A$ is not excellent.
This completes the proof of item~3.

By Equation~\ref{16.5.15}.0, $B/xB = A/xA = R^*/xR^*$. Hence
$\m_A = (x, y)A$ is the unique prime of $A$ lying over
$\n = (x, y)B$. Thus  item~4 holds and for item 5 we see that $x$ is not in any
height-two prime ideal
 of $B$.

To complete the proof of  item~5,  it remains to consider $P \in \Spec B$ with
$x \not\in P$ and $\hgt P>1$.
By Proposition~\ref{16.5.15}.4, we have $x^n  \not\in PR^*$ for each $n \in \N$.
Thus $\hgt(PR^*) \le 1$.
Since $A \hookrightarrow R^*$ is
faithfully flat, $\hgt(PA) \le 1$. Let $P'$ be a height-one prime ideal of $A$ containing $PA$.
Since $\dim B = 3$, $\hgt P > 1$ and $x \not\in P'\cap B$, it follows that $P = P' \cap B$.
If
$p \notin P$, then  Proposition~\ref{16.5.17} implies that  $A_{P'}=B_P$.
Since $P'$ is a height-one prime ideal of $A$,  it follows that
$P$ is a height-one prime ideal of $B$.

Now suppose that $p_i\in P$ for some $i$. Then  $p_iR^*$ is a
height-one prime ideal contained in
$ PR^*$
 and so $p_iR^*= PR^*$. Hence $P$ is squeezed between $p_iB$ and $Q_i=p_iR^*\cap B\ne (x,y)B$.
Since
 $\dim B=3$, either $P$ has height one or $P=Q_i$ for some $i$.  This completes the proof of item 5.

For item 6, we show that each $Q_i$ is not finitely generated by showing
 for each $n \ge 0$, that
$f_{n+1} \not\in (p_i, f_n)B$.
By Equation~\ref{16.5.1}.b, we   have  $\tau_n = c_{n+1}x + x\tau_{n+1}$,  and hence
$f_n =  xf_{n+1} ~+~pxc_{n+1}$.
Assume that $f_{n+1} \in  (p_i, f_n)B$. Then
$$
(p_i, f_n)B = (p_i, xf_{n+1} +pxc_{n+1})B ~ \implies ~ f_{n+1} = ap_i+b(xf_{n+1} + pxc_{n+1}),
$$
for some $a, b \in B$. Thus $f_{n+1}(1 - xb) \in p_iB$. Since $1-xb$ is a unit of $B$,
it follows that $f_{n+1} \in p_iB$, and thus
$f_{n+1} \in p_iB_{n+r}$, for some $r \ge 1$. The relations $f_t = xf_{t+1} + pxc_{t+1}$,
for each $t \in \N$, imply that
$$
f_{n+1} = xf_{n+2} + pxc_{n+2} = x^2f_{n+3} + px^2c_{n+3} + pxc_{n+2} = \cdots =
x^{r-1}f_{n+r} + p\alpha,
$$
where $\alpha \in R$. Thus  $x^{r-1}f_{n+r} \in (p, f_{n+1})B_{n+r}$. Since $f_{n+1} \in p_iB_{n+r}$,
we have  $x^{r-1}f_{n+r} \in p_iB_{n+r}$. This implies $f_{n+r} \in p_iB_{n+r}$, a contradiction
because the ideal $(p_i, f_{n+r})B_{n+r}$ has height two.
We conclude that   $Q_i$ is
not finitely generated.

Since $B$ is a UFD, the height-one primes of $B$ are principal and since the maximal
ideal of $B$ is two-generated, every nonfinitely generated prime ideal of $B$
has height two and thus
is in the set   $\{Q_1, \ldots, Q_m\}$.  This completes the proof of item~6.

For item~7,  the chain $(0) \subset xB \subset (x,y)B = \m_B$ is saturated and has length two,
while the
chain $(0) \subset p_1B \subset Q_1 \subset \m_B$ is saturated and has length three.
\end{proof}

\medskip

\begin{remark} With the notation of Examples~\ref{16.5.1} and Theorem~\ref{16.5.2} we obtain the following additional details about the prime ideals  of $B$.
\begin{enumerate}
\item If  $P \in \Spec B$ is nonzero and  nonmaximal, then
$\hgt(PR^*) = 1$ and $\hgt(PA) = 1$. Thus
every  nonmaximal prime of $B$ is contained in a nonmaximal
prime of $A$.
\item
If $P \in \Spec B$ is such that $P \cap R = (0)$, then $\hgt(P) \le 1$ and
$P$ is principal.

\item If $P \in \Spec B$, $\hgt P =  1$ and $P \cap R \ne 0$,
then $P = (P \cap R)B$.

\item
Let $p_i$ be one of the prime
factors of $p$. Then $p_iB$ is prime in $B$. Moreover
the  ideals
$p_iB$ and $Q_i := p_iA \cap B =    (p_i, f_1, f_2, \ldots )B $ are the
only nonmaximal  prime ideals of $B$ that contain $p_i$. Thus they
are the only prime ideals of $B$ that lie  over $p_iR$ in $R$.

\item The constructed ring $B$ has Noetherian spectrum
\end{enumerate}\end{remark}

\begin{proof} For the proof of item~1, if $P = Q_i$ for some $i$, then $PR^* \subseteq p_iR^*$
and $\hgt PR^* = 1$. If $P$ is not one of the $Q_i$, then by Theorem~\ref{16.5.2} $P$ is a
principal height-one prime and $\hgt PR^* = 1$. Since $A$ is Noetherian and local, $R^*$ is
faithfully flat over $A$ and hence $\hgt PA = 1$.

The proof of item 1 is contained in the proof of item 5 of Theorem~\ref{16.5.2}.

For item~2, $\hgt P \le 1$ because
the field of fractions $K(f)$ of $B$ has transcendence degree one
over the field of fractions $K$ of $R$.  Since $B$ is a UFD, $P$ is principal.

For item~3, if $x \in P$, then $P = xB$ and the statement
is clear. Assume $x \not\in P$. By Remark~\ref{16.5.11},
$B[1/x]$ is a localization of $B_n$, and so $\hgt(P\cap B_n)=1$ for all integers $n \ge 0$.
Thus $(P \cap R)B_n = P \cap B_n$, for each $n$,   and so $P = (P \cap R)B$.

For item~4,  each $p_iB$ is prime by Proposition~\ref{16.5.15}.3. By Theorem~\ref{16.5.2},
$\dim B=3$ and
the $Q_i$ are the only height-two primes  of $B$. Since for $i \ne j$, the
ideal $p_iR + p_jR$ is $\m_R$-primary,  it follows that $p_iB + p_jB$  is $\n$-primary, and hence
$p_iB$ and $Q_i$ are the only nonmaximal prime ideals of $B$  that contain $p_i$.

Item~5 follows from Theorem~\ref{16.5.2}, since the prime spectrum is Noetherian if it satisfies the ascending chain condition and if, for each finite set in the spectrum,
there are only finitely many points 
minimal with respect to containing all of them. Thus the proof is complete.
\end{proof}

\begin{remark} \label{coherence} Rotthaus and Sega prove that the rings $B$ of
Theorem~\ref{build}, Theorem~\ref{16.5.2}, and Theorem~\ref{16.5.4t}  are coherent and
regular in the sense that every finitely generated submodule of a free module has a finite
free resolution \cite{RS}. For the ring
$B = \bigcup_{n=1}^\infty B_n$ of these constructions, it is stated in \cite{RS}
that $B_n[1/x] = B_{n+k}[1/x] = B[1/x]$
and that $B_{n+k}$ is generated over $B_n$ by a single element for all positive integers
$n$ and $k$.  This is not correct for the local rings $B_n$. However, if instead of using
the localized polynomial rings $B_n$ and their union $B$ of the construction for these
theorems, one uses the underlying polynomial rings $U_n$ and their union $U$ defined in
Theorem~\ref{build},
then one does have that $U_n[1/x] = U_{n+k}[1/x] = U[1/x]$ and that $U_{n+k}$ is
generated over $U_n$
by a single element for all positive integers
$n$ and $k$.
\end{remark}

\vskip 3pt

We use the following lemma.

\begin{lemma} \label{prime} Let the notation be as in  Examples~\ref{16.5.1} and  Theorem~\ref{16.5.2}. \begin{enumerate}
\item For every
 element $c \in \m_R \setminus xR$ and every $t \in \N$, the element $c + x^tf$ is a prime
element of the UFD $B$.
\item For every fixed  element $c \in \m_R \setminus xR$, the set $\{ c+x^tf\}_{t\in\N}$
consists of infinitely many nonassociate prime elements of $B$, and so
 there exist infinitely many distinct height-one primes of $B$ of the form $(c + x^tf)B$.
\end{enumerate}
\end{lemma}
 \begin{proof}
 For the first item, since $f=p\tau$,
Equation~\ref{16.5.1}.b implies that
  $$
  f_r=pc_{r+1}x +xf_{r+1}.
  $$
  In
$B_0=k[x,y,f]_{(x,y,f)}, $ the polynomial $c+x^tf$
  is linear  in the variable $f=f_0$ and the coefficient $x^t$ of $f$
is relatively prime to the constant term
  $c$.   Thus   $c+x^tf$ is irreducible in $B_0$.
  Since $f=f_0=pc_1x+xf_1$ in $B_1=k[x,y,f_1]_{(x,y,f_{1})}$,   the
polynomial $c+x^tf=c + x^tpc_1x+x^{t+1}f_1$  is   linear
in the variable $f_1$ and the coefficient $x^{t+1}$ of $f_1$ is relatively prime to the constant term
  $c$.  Thus   $c +  x^tf$ is irreducible in $B_1$.
To see that this pattern continues, observe that  in $B_2$, we have
  \begin{equation*}
  \aligned
  f~=~&pc_1x+xf_1~=~pc_1x+pc_2x^2+x^2f_2\implies \\
  &c +x^tf
 ~=~c + pc_1x^{t+1}+pc_2x^{t+2} +x^{t+2}f_2,
  \endaligned
  \end{equation*}
 a  linear polynomial in the variable $f_2$.   Thus   $c +x^tf$ is irreducible in $B_2$ and a
 similar argument shows that  $c +x^tf$ is irreducible in $B_r$ for each positive integer $r$.
 Therefore for each $t \in\N$,  the element  $c + x^tf$ is
 prime in $B$.

For item 2,  observe  that $(c + x^tf)B \ne (c + x^mf)B$, for positive
integers $t > m$.  If $(c + x^tf)B = (c+ x^mf)B :=
q$, a height-one prime ideal of $B$, then
$$
(x^t - x^m)f~  =~ x^m(x^{t-m} - 1)f ~\in ~q.
$$
Since  $c \notin xB$ we have $q \ne xB$. Thus
$x^m \notin q$. Also $x^{t-m}- 1$ is a unit of $B$. It follows that $f \in q$ and thus
$(c, f)B \subseteq q$.

By Remark~\ref{16.5.11}, $B[1/x]$ is a localization of $R[f] = S$,   and
$x \notin q$  implies  that $B_q = S_{q \cap S}$. This is a
contradiction since the ideal $(c, f)S$ has
height two. Thus there exist infinitely many distinct height-one primes of
the form $(c + x^tf)B$.
\end{proof}

Lemma~\ref{lost} is useful for giving a  more precise description of $\Spec B$ for $B$
as in Examples~\ref{16.5.1}. For each nonempty
finite subset $H$   of $\{Q_1, \ldots, Q_m\}$, we show there exist infinitely many
height-one prime ideals contained in each $Q_i  \in H$, but not contained in $Q_j$
if $Q_j \notin H$. 

\begin{lemma} \label{lost} Let the notation be as in Theorem~\ref{16.5.2}.
Let $G$ be a nonempty subset of $\{1, \ldots , m\}$, and let
$H = \{Q_i ~|~ i \in G \}$. Let $p_G = \prod \{ p_i ~|~i \in G \}$. Then
for each $t \in \N$, we have
\begin{enumerate}
\item
$(p_G + x^tf)B$ is a prime ideal of $B$ that is lost in $A$.
\item
$(p_G^2 + x^tf)B$ is a prime ideal of $B$ that is not lost in $A$.
\end{enumerate}
The sets $\{(p_G + x^tf)B \}_{t \in \N}$ and $\{(p_G^2  + x^tf)B \}_{t \in \N}$ are both
infinite.
Moreover, the prime ideals in both item 1 and item 2 are contained in each $Q_i$ such
that $Q_i \in H$, but are  not contained in  $Q_j$
if $Q_j \notin H$.
\end{lemma}

\begin{proof}
For item 1, we have
\begin{equation*}
(p_G + x^tf)A \cap B = p_G(1 + x^t\tau \prod_{j \notin G}p_j)A \cap B =
p_GA \cap B  = \bigcap_{i \in G}
Q_i. \tag{\ref{lost}.1}
\end{equation*}
Thus each prime ideal of $B$ of the form $(p_G + x^tf)B$ is lost in $A$ and $R^*$.
 By the second item of Lemma~\ref{prime},  there exist
infinitely many height-one primes $(p_G + x^tf)B$
of $B$
that are lost in  $A$ and $R^*$.

For item 2, we have
\begin{equation*} \aligned
(p_G^2 + x^tf)A \cap B & = (p_G^2 +x^tp_G(\prod_{j \notin G}p_j)\tau)A  \cap B\\ &
 =  p_G(p_G + x^t(\prod_{j \notin G}p_j) \tau)A \cap B \subsetneq p_GA \cap B = \bigcap_{i \in G}Q_i.
\endaligned \tag{\ref{lost}.2}
\end{equation*}
The strict inclusion is because $p_G + x^t (\prod_{j \notin G}p_j) \tau \in \m_A$.
This implies that prime ideals of $B$
of form $(p_G^2 + x^tf)B$ are not lost.
By Lemma~\ref{prime} there are infinitely many distinct prime ideals of that form.

The ``moreover'' statement  for  the prime ideals in item 1 follows from Equation~\ref{lost}.1.
Equation~\ref{lost}.2 implies that  the prime ideals in item 2 are contained
in each $Q_i \in H$. For $j \notin G$, if $p_G^2 + x^tf \in Q_j$, then $p_j + x^tf \in Q_j$ implies
that $p_G^2  - p_j  \in Q_j$ by subtraction. Since $p_j \in Q_j$, this would imply that
$p_G^2 \in Q_j$, a contradiction. This completes the proof of Lemma~\ref{lost}.
\end{proof}

\begin{remark}  \label{16.555}   With the notation of Examples~\ref{16.5.1},   consider  the
birational inclusion $B \hookrightarrow A$ and
the faithfully flat map $A \hookrightarrow R^*$.   The  following statements hold concerning
the inclusion maps $R \hookrightarrow B \hookrightarrow A \hookrightarrow R^*$,  and the
associated maps  in
the opposite direction of their spectra.

\begin{enumerate}
\item  The map $\Spec R^* \to \Spec A$ is surjective, while the
maps $\Spec R^* \to \Spec B$ and $\Spec A \to \Spec B$ are
not surjective. All the induced maps to $\Spec R$ are surjective since
the map $\Spec R^* \to \Spec R$
is surjecive.

\item
By Lemma~\ref{lost}, each of the prime ideals $Q_i$ of $B$ contains infinitely many
height-one primes of $B$ that are the contraction of prime ideals of
$A$ and infinitely many that are not.

Since an ideal contained in a finite union of prime
ideals is contained in one of the prime ideals by   \cite[Prop. 1.11, page~8]{AM},
there are infinitely many non-associate prime elements of  the UFD
$B$ that are not contained in the union $\bigcup_{i=1}^m Q_i$. We observe
that for each prime element $q$ of $B$  with $q \notin \bigcup_{i=1}^m Q_i$ the ideal
 $qA$ is contained in a height-one prime $\q$ of $A$
and $\q \cap B$ is properly contained in $\m_B$ since $\m_A$ is the unique
prime ideal of $A$ lying over $\m_B$. Hence $\q \cap B = qB$. Thus each $qB$ is
contracted from $A$ and $R^*$.

In the four-dimensional example $B$ of
Theorem~\ref{16.5.4t},  each height-one prime of $B$ is
contracted from $R^*$, but there are infinitely many height-two primes of
$B$ that are lost in $R^*$, i.e., are not contracted from $R^*$; see Section~\ref{5}.

\item  Among the prime ideals of the domain $B$  of Example~\ref{16.5.1} that are
not contracted from
$A$ are the $p_iB$.  Since
$p_iA \cap B = Q_i$
properly contains $p_iB$, the prime ideal   $p_iB$ is lost in $A$.

\item Since $x$ and $y$ generate the maximal ideals of $B$ and $A$,
and since  $B$ is integrally closed, a version
of Zariski's Main Theorem \cite{Peskine}, \cite{Ev}, implies
that $A$ is not essentially
finitely generated as a $B$-algebra.
\end{enumerate}
\end{remark}

Using the information above, we display  below a picture of $\Spec(B)$ in the case $m=2$.
\vskip 10 pt

\setbox4=\vbox{\hbox{%
\hskip 20pt
     \rlap{\kern  2.050in\lower 0.166in\hbox to .3in{\hss $\m_B:=(x,y)B$\hss}}%
        \rlap{\kern  1.850in\lower .666in\hbox to .3in{\hss $Q_1$ \hss}}%
        \rlap{\kern  3.450in\lower .666in\hbox to .3in{\hss $Q_2$ \hss}}%

   \rlap{\kern  .250in\lower 1.266in\hbox to .3in{\hss $xB\in\boxed{\text{NOT Lost}}$\hss}}%
   \rlap{\kern  1.350in\lower 1.266in\hbox to .3in{\hss $\boxed{\text{NL}}$\hss}}%
   \rlap{\kern  1.750in\lower 1.266in\hbox to .3in{\hss $\boxed{\text{L}}$\hss}}%
  \rlap{\kern  2.450in\lower 1.266in\hbox to .3in{\hss $\boxed{\text{NL}}$\hss}}%
   \rlap{\kern  2.850in\lower 1.266in\hbox to .3in{\hss $\boxed{\text{L}}$\hss}}%
 \rlap{\kern  3.550in\lower 1.266in\hbox to .3in{\hss $\boxed{\text{NL}}$\hss}}%
   \rlap{\kern  3.950in\lower 1.266in\hbox to .3in{\hss $\boxed{\text{L}}$\hss}}%

\rlap{\kern  2.050in\lower 1.666in\hbox to .3in{\hss $(0)$\hss}}%
     \rlap{\special{pa 2100 180}    \special{pa 350 1100}    \special{fp}}%
\rlap{\special{pa 2100 180}    \special{pa 1900 550}    \special{fp}}%
\rlap{\special{pa 2100 180}    \special{pa 3400 550}    \special{fp}}%
     \rlap{\special{pa 1850 750}    \special{pa 1500 1100}    \special{fp}}%
   \rlap{\special{pa 1850 750}    \special{pa 1850 1100}    \special{fp}}%
    \rlap{\special{pa 1850 750}    \special{pa 2550 1100}    \special{fp}}%
    \rlap{\special{pa  1850 750}    \special{pa 2900 1100}    \special{fp}}%
     \rlap{\special{pa 3450 750}    \special{pa 2550 1100}    \special{fp}}%
    \rlap{\special{pa  3450 750}    \special{pa 2900 1100}    \special{fp}}%
     \rlap{\special{pa 3450 750}    \special{pa 4050 1100}    \special{fp}}%
    \rlap{\special{pa  3450 750}    \special{pa 3600 1100}    \special{fp}}%
  \rlap{\special{pa 350 1360}    \special{pa 2100 1570}    \special{fp}}%
    \rlap{\special{pa 1450 1360}    \special{pa 2100 1570}    \special{fp}}%
     \rlap{\special{pa 1850 1360}    \special{pa 2100 1570}    \special{fp}}%
    \rlap{\special{pa 2550 1360}    \special{pa 2100 1570}    \special{fp}}%
     \rlap{\special{pa 3450 1360}    \special{pa 2100 1570}    \special{fp}}%
    \rlap{\special{pa 4050 1360}    \special{pa 2100 1570}    \special{fp}}%
        \rlap{\special{pa 2850 1360}    \special{pa 2100 1570}    \special{fp}}%
   }} \box4 \vskip 10 pt

\centerline{Diagram~\ref{16.555}.0}
\noindent
{\bf Comments on Diagram~\ref{16.555}.0.}
Here we have $Q_1 = p_1R^* \cap B$ and $Q_2 = p_2R^* \cap B$, and each box
represents an infinite set of height-one prime ideals. We label a box ``NL'' for
``not lost'' and ``L'' for ``lost''. An argument similar to that given for the Type   I
primes in
Example~\ref{16.5.3de}  shows that  the height-one primes $q$
such that $q \notin Q_1 \cup Q_2$ are not lost. That the other boxes are infinite follows
from Lemma~\ref{lost}.

\section{A 4-dimensional prime spectrum}  \label{5}

In Example~\ref{16.5.4de}, we present  a 4-dimensional example analogous to
Example~\ref{16.5.3de}.

\begin{example}\label{16.5.4de} Let $k$ be a field, let $x$, $y$  and $z$ be indeterminates
over $k$.   Set
$$
R : ~~=~~  k[x, y,z]_{(x,y,z)}  \qquad \text{and} \qquad R^*:~~=~~k[y, z]_{(y,z)}[[x]],
$$
and let $\m_R$ and $\m_{R^*}$ denote the maximal ideals of $R$ and $R^*$, respectively.
The power series ring   $R^*$ is  the $xR$-adic completion of $R$.  Consider  $\tau$  and $\sigma$  in
$xk[[x]]$
$$
\tau ~:=~ \sum_{n=1}^\infty c_nx^n \quad \text{and} \quad \sigma~:= ~ \sum_{n=1}^\infty d_nx^n,
$$
where the $c_n$ and $d_n$ are in $k$ and $\tau$ and $\sigma$ are
algebraically independent  over $k(x)$.  Define
$$
f~ :=~ y\tau + z\sigma \quad \text{and} \quad  A~:=~ A_f = R^* \cap k(x,y,z,f),
$$
that is, $A$ is the intersection domain associated with $f$.
For each integer $n \ge 0$, let $\tau_n$ and $\sigma_n$ be the $n^{\text{th}}$
endpieces of $\tau$ and $\sigma$ as in Equation~\ref{16.5.1}.a. Then the
$n^{\text{th}}$ endpiece of $f$ is  $f_n = y\tau_n + z\sigma_n$.
As in  Equation~\ref{16.5.1}.b, we have
$$
\tau_n ~=~x\tau_{n+1}~+~ xc_{n + 1} \quad \text{ and } \quad \sigma_n ~=~ x\sigma_{n+1}~+~xd_{n + 1},
$$
where $c_{n+1}$ and $d_{n+1}$ are in the field  $k$. Therefore
\begin{equation*}
\aligned
f_n ~ =~ y\tau_n~+~ z\sigma_n ~&=~ yx\tau_{n+1}+yxc_{n+1}  ~+~zx\sigma_{n+1} ~+~  zxd_{n+1} \\
      &=~ xf_{n+1} ~+yxc_{n+1} ~+~zxd_{n+1}.
\endaligned
\tag{\ref{16.5.4de}.1}
\end{equation*}
The approximation domains $U_n$, $B_n$, $U$  and $B$ for $A$ are as follows:
$$\aligned
\text{For } n~\ge ~0, \quad
U_n ~&:= ~k[x,y,z,f_n] \quad &B_n ~:= ~k[x,y,z,f_n]_{(x,y,z,f_n)} \quad \\
U~&:=~  \bigcup_{n=0}^\infty U_n \quad \text{and} \quad &B~:=~ B_f ~=~ \bigcup_{n=0}^\infty B_n.
\endaligned
$$
Thus $B$  is  the directed union of  4-dimensional localized polynomial rings.  It follows that
$\dim B \le 4$.

The rings $A$ and $B$ are constructed inside the
  intersection domain $A_{\tau, \sigma}:= R^* \cap k(x,y,z,\tau,\sigma)$.
By \cite[Prop. 4.1]{noehom} or \cite[Theorem~9.2]{powerbook},  
the domain $A_{\tau, \sigma}$ is Noetherian and equals  its
 approximation domain $B_{\tau, \sigma}$.  Here $B_{\tau, \sigma}$ is the
nested union of the regular local domains 
$B_{\tau,\sigma, n} = k[x,y,z, \tau_n, \sigma_n]_{(x,y,z, \tau_n, \sigma_n)}$. 
By Theorem~\ref{build}, the extension $T := R[\tau, \sigma] \hookrightarrow R^*[1/x]$ is
flat. It follows that $A_{\tau, \sigma}$ is a
3-dimensional RLR that is a directed union of 5-dimensional RLRs.

\end{example}

Before we list and establish the other properties of Example~\ref{16.5.4de}
in Theorem~\ref{16.5.4t}, we discuss the Jacobian ideal of a map and its
relation to flatness.

\medskip
\begin{discussion} \label{flatpoly}
Let $R$ be a Noetherian ring, let  $m$ and $n$ be positive integers,  let
$z_1, \ldots, z_n$ be indeterminates over $R$ and let $f_1,\dots,f_m$ be
polynomials in $R[z_1, \ldots, z_n]$  that are algebraically
independent over $R$. Let
\begin{equation*}
S ~~:= ~~
R[f_1, \ldots, f_m] ~~~ \overset{\varphi}\hookrightarrow ~~~ R[z_1, \ldots, z_n]~~ =:~~ T,
\tag{\ref{flatpoly}.0}
\end{equation*}
be the inclusion map.

We define
the {\it Jacobian ideal} $J$  of the extension $S \hookrightarrow T$
to be the ideal  of $T$ generated by the
$m \times m$ minors of the $m \times n$ matrix $
\mathcal J$ defined as follows:
$$
\mathcal J :=\left( \frac{\partial f_i}{\partial z_j}\right)_{1 \le i \le m,\,\, 1 \le j \le n}.
$$

For the extension $\varphi : S \hookrightarrow T$,
the {\it nonflat locus } of $\varphi$ is  the set $\mathcal F$, where
$$
\mathcal F  :=  \{Q \in \Spec(T)~~|~~ \text{the map} \quad \varphi_Q: S \to T_Q
\text{ is not flat } \}.
$$
The nonflat locus of $\varphi$ is a closed subset of $\Spec T$, \cite[Theorem~24.3]{M}.
We say that  an ideal $F$ of $T$  {\it defines the nonflat locus} of $\varphi$
if $F$ is such that for every $Q\in\Spec(T)$, we have $F\subseteq Q  \iff $
 the associated
map $\varphi_Q: S_{Q\cap S}\to T_Q$ is not flat.

\end{discussion}

\begin{proposition} \label{17.45p} $\phantom{i}$  \cite[Propositions~2.4.2, 2.7.2]{fez97}
With  notation as in  Discussion~\ref{flatpoly},  let  $Q \in \Spec T$ and
consider $\varphi_Q : S \to   T_Q $. Then
\begin{enumerate}
\item
$\varphi_Q $ is flat if and only if,  for
each prime ideal $ P \subseteq Q$ of $T$,
we have  $\text{ht}(P)\ge \text{ht}(P\cap S)$.
\item If $Q$ does not contain $J$, then $\varphi_Q $ is flat. Thus $J\subseteq F$.
\end{enumerate}
\end{proposition}

We use the following proposition concerning flatness to justify Example~\ref{16.5.4de}.

\begin{proposition} \label{nfl} With the notation of Example~\ref{16.5.4de}, we have
\begin{enumerate}
\item \label{yzJ} 
For the extension $\varphi: S=R[f]\hookrightarrow T=R[\tau,\sigma]$,  
the Jacobian ideal $J$ is the ideal $(y,z)T$.  Thus the nonflat locus $F$ of $\varphi$ contains $J$.

\item  For every  $P\in\Spec(R^*[1/x])$, the ideal  $(y,z)R^*[1/x] \nsubseteq P\iff$
the map $B_{P\cap B}\hookrightarrow (R^*[1/x])_P$ is flat.
Thus the ideal $(y,z)R^*[1/x]$ defines the nonflat locus of
the map $B\hookrightarrow R^*[1/x]$.
\item For every height-one prime ideal $\p$ of $R^*$, we have  $\hgt(\p\cap B) \le 1$.
\item For every prime element $w$ of $B$, $wR^*\cap B=wB$.
\end{enumerate}
\end{proposition}
\begin{proof}  For item 1, 
 the Jacobian ideal is the ideal of $T$ 
generated by the $1\times 1$ minors of the matrix $(y\phantom{X} z)$ by Discussion~\ref{flatpoly}, and so $J=(y,z)T$. 
By Proposition~\ref{17.45p}.2, $(y,z)T\subseteq F$.

The two statements of item~2 are equivalent by the
definition of nonflat locus in Discussion~\ref{flatpoly}. To compute the nonflat locus of $B \hookrightarrow R^*[1/x]$,
we use that $T: = R[\tau, \sigma] \hookrightarrow R^*[1/x]$ is flat as
noted in Example~\ref{16.5.4de}.
 Let  $P \in \Spec(R^*[1/x])$  and let $Q := P \cap T$.
The map  $B  \hookrightarrow  R^*[1/x]_P$ is flat $\iff$ the composition
$$\aligned k[x,y,z,f]&\hookrightarrow k[x,y,z, \tau, \sigma]\hookrightarrow
R^*[1/x]_P \text{ is flat}\iff \\
S:=&k[x,y,z,f]\overset{\varphi}{\hookrightarrow} T_Q =k[x,y,z,\tau, \sigma]_Q \text{ is flat}.
\endaligned
$$
From item~1,  the Jacobian ideal of the extension $S \hookrightarrow T$ is the ideal    $J=(y,z)T$.
Since $(y,z) T\cap S=(y,z,f)S$ has height 3, $\varphi_Q$ is not flat
for every $Q\in\Spec(T)$ such that $(y,z)T\subseteq Q$.
Thus the nonflat locus of $B\hookrightarrow R^*[1/x]$ is $(y,z)R^*[1/x]$ as stated in item~2.

For item 3, let $\p$ be a height-one prime of $R^*$.
 Since $\p$ does not  contain $(y,z)R^*$, the map  $B_{\p\cap B}\hookrightarrow (R^*)_{\p}$
is faithfully flat.
 Thus $\hgt(\p\cap B)\le 1$.  This establishes  item~3.

Item~4 is clear if $wB = xB$.   Assume that
 $wB \ne xB$ and let    $\p$ be a height-one prime ideal of $R^*$ that contains $wR^*$.  Then
$\p R^*[1/x] \cap R^* = \p$,
 and by item~3, $\p \cap B$ has  height  at most one.  We have
$\p \cap B\supseteq wR^*\cap B\supseteq wB$. Thus item~4 follows.
\end{proof}

Next  we prove a proposition about homomorphic
images of the constructed ring $B$. This result  enables us in Corollary~\ref{homimB-z}
to relate the ring $B$ of
Example~\ref{16.5.4de} to the ring $B$ of Example~\ref{16.5.3de}.

\begin{proposition}\label{homimB} Assume the notation of Example~\ref{16.5.4de}, and
let $w$ be a prime element of $R=k[x,y,z]_{(x,y,z)}$ with $wR\ne xR$.
Let $\pi: R^*\to R^*/wR^*$ be the natural
homomorphism, and let $\overline{\phantom{x}}$ denote  image  in $R^*/wR^*$.
 Let $B'$ be the approximation
domain formed
by considering $\overline R$   and the endpieces $\overline f_n$  of $\overline f$,
 defined analogously to Equation~\ref{16.5.1}.a.
That is, $B'$ is defined  by setting
 $$U'_n ~= ~\overline R[\overline f_n],~ B'_n~ = ~({U_n'})_{\n'_n}, ~ U'~=  ~\bigcup_{n=1}^\infty
U'_n,  ~~\text{and } ~~ B'~ = ~\bigcup_{n=1}^\infty B'_n,$$
 where $\n'_n$ is the maximal ideal of $U'_n$ that contains $\overline f_n$ and the image of $\m_R$.
Then $B'= \overline B$.
\end{proposition}

\begin{proof}
By Proposition~\ref{16.5.15}.3, $wB$ is a prime ideal of $B$. By Proposition~\ref{nfl}.3,
$wR^* \cap B = wB$. Hence $\overline B = B/(wR^* \cap B) = B/wB$.
 We have
$$
\overline{R}/x\overline{R} ~=~ \overline{B}/x\overline{B} ~=~ \overline{R^*}/x\overline{R^*},
$$
and the ring $\overline{R^*}$ is the $(\overline x)$-adic completion of $\overline R$.
Since the ideal $(y,z)R$ has height 2 and the kernel of $\pi$ has height~1,
at least one of $\overline y$ and $\overline z$ is  nonzero. Since
$\tau$ and $\sigma$ are algebraically independent
over $k(x,y,z)$, the element   $\overline f=\overline y\cdot\overline \tau+\overline z\cdot\overline \sigma$ of
the integral domain $\overline B$ is
transcendental over $\overline R$.
Similarly the endpieces  $\overline f_n$ are  transcendental over $\overline  R$.
The fact that $\overline R^*$
may fail to be an integral domain does not affect the algebraic  independence of these
elements  that are inside the integral domain $\overline B$.

By Proposition~\ref{16.5.15}.2 and Remarks~\ref{locBsame}.2, we have $U_n[1/x] = U[1/x]$, and thus $wU \cap U_n = wU_n$
for each $n \in \N$. Since $B_n$ is a localization of $U_n$, we also have
$wB \cap B_n = wB_n$. Since $wR^* \cap B = wB$, it follows that $wR^* \cap B_n = wB_n$.
Thus we have
$$
\overline R~\subseteq ~  \overline {B_n} ~=~ B_n/wB_n ~ \subseteq ~ ~
\overline B ~=~ B/wB ~ \subseteq ~ \overline{R^*} ~= ~R^*/wR^*.
$$
We conclude that $\overline B  = \bigcup_{n=0}^\infty \overline{B_n}$. Since  $B'_n=\overline B_n$,
we have  $B'=\overline B$.
\end{proof}

\begin{corollary}\label{homimB-z} The  homomorphic image $B/zB$ of the
 ring $B$ of
Example~\ref{16.5.4de} is isomorphic to the three-dimensional
ring $B$ of Example~\ref{16.5.3de}.
\end{corollary}

\begin{proof}
Assume the notation of Example~\ref{16.5.4de} and Proposition~\ref{homimB} and let $w=z$.
We show that the ring $B/zB\cong C$, where $C$ is the ring called $B$ in Example~\ref{16.5.3de}.
 By Proposition\ref{homimB}, we have  $B'= B/zB$, where $B'$ is the approximation domain
over $\overline R=R/zR$ using the element $\overline f$, transcendental over $\overline R$.
 Let $R_C$ denote
the base ring  $k[x,y]_{(x,y)}$
for $C$ in Example~\ref{16.5.3de}, and let $\psi_0: \overline R \to R_C$
denote the $k$-isomorphism defined by $\overline x \mapsto x$ and $\overline y
\mapsto y$.
Then $\psi_0$ extends to an isomorphism $\psi:(\overline R)^*\to (R_C)^*$
that agrees with $\psi_0$ on $\overline R$ and such that $\psi(\overline \tau)=\tau$.
Furthermore $\psi(\overline f)=\psi(\overline y\cdot\overline \tau+\overline z\cdot\overline
\sigma)=
y\tau$, which is the transcendental element $f$
used in the construction of $C$.
Thus $\psi$ is an isomorphism from $\overline B=B/zB$ to $C$,
the ring constructed in Example~\ref{16.5.3de}.
\end{proof}

In the proof of Theorem~\ref{16.5.4t}, we use the following proposition 
regarding a birational extension of a Krull domain.

\begin{proposition} \label{6.2.9fp}   Let  $S \hookrightarrow T$ be a birational extension
of commutative rings, where $S$ is a Krull domain  and each height-one prime 
ideal of $S$  is 
contracted from $T$. Then $S=T$.
\end{proposition}
\begin{proof} Recall that $S$ is Krull implies that   $S=\cap \{S_\p\,|\, \p$ is a height-one prime ideal of $S\}$. We show that $T\subseteq S_\p$, for each height-one prime ideal of $S$. Since $\p$ is contracted from $T$, there exists a prime ideal $\q$ of $T$ such
that $\q\cap S=\p$. Then $S_\p\subseteq T_\q$ and $T_\q$ birationally dominates $S_\p$.  Since $S_\p$ is a DVR, we have $S_\p=T_\q$.  Therefore $T\subseteq S_\p$, for each $\p$.  It follows that  $T= S$.
\end{proof}

We record in Theorem~\ref{16.5.4t} properties of the ring $B$ and its prime spectrum.

\begin{theorem} \label{16.5.4t} As in Example~\ref{16.5.4de},  let
$R : =k[x, y,z]_{(x,y,z)}$ with $k$  a field, let $x$, $y$  and $z$ be indeterminates, and  
let $R^*:=k[y, z]_{(y,z)}[[x]]$, the $xR$-adic completion of $R$. 
 Let $\tau$ and $\sigma\in xk[[x]]$ be
algebraically independent  over $k(x)$.
\quad Set
$
f:=y\tau + z\sigma$, \\ $A:=R^* \cap k(x,y,z,f),
$ and
$B:=\bigcup_{n=0}^\infty B_n=\bigcup_{n=0}^\infty k[x,y,z,f_n]_{(x,y,z,f_n)}$ as in~(\ref{16.5.4de}.2). 
Let $Q := (y, z)R^* \cap B$. Then
\begin{enumerate}
\item The rings $A$ and $B$ are equal.

\item The ring $B$ is a four-dimensional non-Noetherian local UFD with maximal
ideal $\m_B = (x,y,z)B$,  and the $\m_B$-adic completion of $B$ is the three-dimensional
RLR $k[[x,y,z]]$.

\item The ring $B[1/x]$ is a  Noetherian regular UFD, the ring
$B/xB$ is a two-dimensional RLR, and for every
nonmaximal prime ideal $P$ of $B$, the ring $B_P$ is an RLR.

\item The ideal $Q$ is the unique prime ideal of $B$ of height 3.

\item
The ideal    $Q$ equals $\bigcup_{n=0}^\infty Q_n$
where $Q_n := (y,z, f_n)B_n$, $Q$ is a nonfinitely generated prime ideal,
and $QB_Q = (y,z,f)B_Q$.

\item There exist infinitely many height-two prime ideals of $B$  not contained
in $Q$ and each of these prime ideals
is contracted from $R^*$.

\item  For certain height-one primes $p$ contained in $Q$, there exist infinitely
many height-two primes between $p$ and $Q$ that  are
 contracted from  $R^*$, and infinitely many that are not contracted from $R^*$.
Hence the map $\Spec R^* \to \Spec B$ is not surjective.

\item Every saturated chain of prime ideals of $B$ has length either 3 or 4, and there
exist saturated chains of prime ideals of lengths both 3 and 4. Thus $B$  is not catenary.

\item
 Each height-one prime  ideal of $B$ is the contraction of a height-one prime ideal  of $R^*$.
 
 \item $B$ has Noetherian spectrum.
\end{enumerate}
\end{theorem}

We prove Theorem~\ref{16.5.4t} below. First, assuming Theorem~\ref{16.5.4t},  we display a
picture of  $\Spec(B)$ and make some remarks.

\setbox4=\vbox{\hbox{%
\hskip 70pt
     \rlap{\kern  1.500in\lower 0.166in\hbox to .3in{\hss $\m_B:=(x,y,z)B$\hss}}%
        \rlap{\kern  2.450in\lower .666in\hbox to .3in{\hss $Q:=(y,z,\{f_i\})B$\hss}}%
\rlap{\kern  -.30in\lower 1.266in\hbox to .3in{\hss $(x,y-\delta z)B\in\boxed{\text{ht. 2},\not\subset Q}$\hss}}%
   \rlap{\kern  .250in\lower 1.866in\hbox to .3in{\hss $xB\in\boxed{\text{ht. 1,}\not\subset Q}$\hss}}%
    \rlap{\kern  1.50in\lower 1.266in\hbox to .3in{\hss
    $\boxed{\text{ht. 2, contr. }R^*}$\hss}}%
\rlap{\kern  3.250in\lower 1.266in\hbox to .3in{\hss
$(y,z)B\in\boxed{\text{ht. 2, Not contr. }R^*}$\hss}}%
   \rlap{\kern 1.7in\lower 1.866in\hbox to .3in{\hss
   $yB,zB\in\boxed{\text{ht. 1,}\subset Q }$\hss}}%
\rlap{\kern  1.500in\lower 2.466in\hbox to .3in{\hss $(0)$\hss}}%
     \rlap{\special{pa 1500 230}    \special{pa 150 1100}    \special{fp}}%
\rlap{\special{pa 1500 230}    \special{pa 2400 550}    \special{fp}}%
\rlap{\special{pa 2400 750}    \special{pa 1500 1100}    \special{fp}}%
    \rlap{\special{pa  2400 750}    \special{pa 3500 1100}    \special{fp}}%
     \rlap{\special{pa 450 1700}    \special{pa 150 1350}    \special{fp}}%
     \rlap{\special{pa 450 1960}    \special{pa 1500 2320}    \special{fp}}%
    \rlap{\special{pa 2000 1960}    \special{pa 1500 2320}    \special{fp}}%
   \rlap{\special{pa 2000 1700}    \special{pa 150 1350}    \special{fp}}%
   }} \box4 \vskip 10 pt

\centerline{Diagram~\ref{16.5.4t}.0}
\noindent
{\bf Comments on Diagram~\ref{16.5.4t}.0.}
A line going from a box at one level to a box at a
higher level indicates that every prime ideal in the lower level box is contained
in at least one prime ideal in the higher level box.  Thus as indicated in the diagram, every
height-one prime $gB$ of $B$ is contained in a height-two prime of $B$ that contains $x$ and so
is not contained in $Q$.
This is obvious if $gB = xB$ and can  be seen by considering minimal primes of $(g,x)B$ otherwise.
Thus $B$ has no maximal saturated chain of length 2.
We have not drawn any lines from the lower level righthand box to higher boxes that
are contained in $Q$ because we are uncertain about what
inclusion relations exist for these primes. We discuss this situation in Remarks~\ref{16.5.4tr}.

\begin{proof}  (of Theorem~\ref{16.5.4t}) For  convenience we prove item~2 first: Since $B$ is a directed union of four-dimensional RLRs, we have $\dim B \le 4$.
By Corollary~\ref{homimB-z} and Theorem~\ref{16.5.2}, $\dim(B/zB)=3$, and so $\dim B\ge 4$. Thus $\dim B=4$.
 By Proposition~\ref{16.5.15}, $B$ is local with maximal ideal $\m_B = (x,y,z)B$ and the
$(x)$-adic completion of $B$ is $R^*$. Thus the $\m_B$-adic completion of $B$ is $k[[x,y,z]]$.
By Krull's Altitude Theorem,  the ring $B$ is not Noetherian \cite[Theorem 13.5]{M}.  The ring $B$ is a UFD by
Proposition~\ref{Bufd}.

For item~1,    the ring $B$ is a UFD by  item~2, and hence  a Krull domain,    and the extension 
$B\hookrightarrow A$ is birational. Thus it suffices to show  that each height-one prime
$P$ of $B$ is the contraction of  a prime ideal of $A$ by Proposition~\ref{6.2.9fp}.

Let $\p$ be a height-one prime ideal of $B$. Then 
$\p R^*\cap B=\p$ by Proposition~\ref{nfl}.4. Also $B\setminus \p$ is a multiplicatively closed subset of $R^*$, and so, if 
$P$ is 
an ideal of $R^*$ maximal with respect to $P\cap (B\setminus \p)=\emptyset$, then $P$ is a prime ideal of $R^*$ and 
$P\cap B=\p$. Then also $P\cap A$ is a prime ideal of $A$ with $(P\cap A)\cap B=\p$, and so $\p$  is contracted from $A$.
Thus $A = B$ as desired for item~1.


For item 3,   the ring
$B[1/x]$ is a Noetherian regular UFD by Proposition~\ref{Bufd}.1.
By Equation~\ref{16.5.15}.0,  we have $R/xR = B/xB$.
Thus  $B/xB$ is a two-dimensional RLR.

For  the last part of item 3,  if $x\notin P$, then $B_P$ is a localization of $B[1/x]$,
which is Noetherian and regular, and so $B_P$ is a regular local ring.
In particular, this proves that
$B_Q$ is a  regular local ring.
If $x\in P$ and $\hgt P=1$, then $P=(x)$ and  $B_{xB}$ is a DVR.
If $x\in P$ and $\hgt(P)=2$, the ideal
$P$ is  finitely generated since $B/xB$  is an RLR. Since $B$ is a UFD
from item~2,  it follows that
$B_P$ is a local UFD of dimension 2 with finitely generated maximal ideal.
Thus $B_P$ is Noetherian by Cohen's Theorem \cite[Theorem~3.4]{M}. This, combined with
$B/xB$ a regular local ring, implies that $B_P$ is a regular local ring.
 Since $\hgt P \le 2$ for every nonmaximal
prime ideal $P$ of $R$ with $x \in P$, this completes the proof of item~3.

For item~4,  since $(y,z)R^*$  is a prime ideal of $R^*$, the ideal $Q = (y,z)R^* \cap B$ is prime.
By Proposition~\ref{16.5.15}, the ideals $yB$ and $(y,z)B$ are prime.
Consider the chain of prime ideals
$$
(0) ~ \subset yB ~ \subset (y,z)B ~ \subset ~ Q ~ \subset ~  \m_B.
$$
The list  $y, z, f, x$ shows that each of the inclusions is strict; for example,  we have $f\in Q\setminus (y,z)B$,
 since $f\notin (y,z) B_n$ for every $n\in\N$. By item~ 2 we have $\hgt\m_B=4$. Thus
$\hgt Q = 3$. This also implies that $(y,z)B$ is a height-two prime ideal of $B$.

For the uniqueness in item 4,  let $P$ be a nonmaximal prime ideal of $B$.  We first consider the case that $x \notin P$. Then, by Proposition~\ref{16.5.15}.4,
$x^n \not\in PR^*$ for each positive integer $n$.  Hence
$PR^*[1/x] \ne R^*[1/x]$.
 Let $P_1$ be a prime ideal of $R^*[1/x]$ such that $P \subseteq P_1$.
If both $y$ and $z$ are in $P_1$,
then $(y,z)R^*[1/x] \subseteq P_1$. Since $(y,z)R^*[1/x]$ is maximal,
we have $(y,z)R^*[1/x] = P_1$.  Therefore, $P\subseteq (y,z)R^*[1/x] \cap B=Q$, and so either
$\hgt(P)\le 2$  or $P=Q$.

Next suppose that $x\notin P$ and $y$ or $z$ is not in $P_1$. By Propositions~\ref{nfl}.1 and \ref{17.45p}.2, the map $\psi: B \to R^*[1/x]_{P_1}$ is flat.
Since  $ \dim R^*[1/x] = 2$ we have $\hgt(P_1) \le 2$. Flatness of  $\psi$
implies  $\hgt(P_1 \cap B) \le 2$; see  \cite[Theorem~9.5]{M}.
Hence $\hgt P \le 2$.   

To complete the proof of  item~4, we consider the case that $x\in P$. We have $\hgt P\le 3$, since $\dim B=4$ and $P$ is not maximal. If $\hgt P\ge 3$,  there exists a  chain of primes of the form 
\begin{equation*}(0)\subsetneq P_1\subsetneq P_2\subsetneq P\subsetneq (x,y,z)B.\tag{\ref{16.5.4t}.1}
\end{equation*}
By Equation~\ref{16.5.15}.0,  $B/xB\cong R/xR$, and so $\dim (B/xB)=2$.  If
$x\in P_2$, then $\hgt P_2\ge 2$ implies that $(0)\subsetneq xB\subsetneq P_2\subsetneq P\subsetneq (x,y,z)B$, a contradiction to  $\dim (B/xB)=2$. Thus $x\notin
P_2$. Since $x\in P$ and $P$ is nonmaximal, we have that $y$ or $z$ is not in $P$.
Hence $y$ or $z$ is not in $P_2$. 

By Equation~\ref{16.5.15}.0, $P$ corresponds to a nonmaximal prime ideal $P'$ of $R^*$ containing $PR^*$.
Let $P_2'$ be a prime ideal of  $R^*$ inside $P'$ that is minimal over $P_2R^*$. If both $y$ and $z$ are in $P_2'$, then, $(x,y,z)R^*\subseteq P'$, a contradiction to $P'$ nonmaximal. 
By Proposition~\ref{16.5.15}.\ref{znpna},  $P_2'$ does not contain $x$. Thus $P_2'\subsetneq P'\subsetneq (x,y,z)R^*$. Also $P_2'=P_2''\cap R^*$, where $P_2''$ is a prime ideal of $R^*[1/x]$, and one of $y$ and $z$ is not an element of $P_2''$.

 By Proposition~\ref{nfl}.2,  the map $\psi: B \to R^*[1/x]_{P_2''}$ is flat. 
 This implies
$\hgt (P_2'')\ge \hgt (P_2''\cap B)\ge \hgt P_2\ge 2$; that is, $\hgt(P_2'') \ge 2$. Also 
 $P_2''$ intersects $R^*$ in $P_2'$, and  so $\hgt P_2'\ge 2$. 
 Thus in 
$R^*$ we have a chain of primes $ P_2'\subsetneq P'\subsetneq (x,y,z)R^*$, where $\hgt P_2'\ge 2$, a contradiction, since $R^*$, a localization of 
$k[y,z][[x]]$, has dimension $3$. 
This proves item~4.

For item~5,  let $Q'= \bigcup_{n=0}^\infty Q_n$, where each $Q_n=(y,z,f_n)B_n$.
Each $Q_n$ is a prime ideal of height 3 in the 4-dimensional
RLR $B_n$. Therefore
$Q'$ is a prime ideal of $B$ of height $\le 3$
that is contained in $Q$.
The ideal $(y,z)B$ is a prime ideal of height 2  by the proof of item~3.
Hence  $\hgt(Q') = 3$ and we have $Q' = Q$.

To show the ideal $Q$ is not finitely generated, we show for
each positive integer $n$ that $f_{n+1} \not\in (y, z, f_n)B$.
By Equation~\ref{16.5.4de}.1, $f_n = xf_{n+1} +yxc_{n+1} + zxd_{n+1}$.  If
$f_{n+1} \in (y,z, f_n)B$, then $f_{n+1} = ay + bz + c(xf_{n+1} +yxc_{n+1} + zxd_{n+1})$,
where $a,b,c \in B$.  This implies $f_{n+1}(1-cx)$ is in the ideal $(y,z)B$. By Proposition~\ref{16.5.15}.1, $x\in\mathcal J(B)$, and so $1 - cx$ is a unit of $B$. 
This implies  $f_{n+1} \in (y,z)B \cap B_{n+1}$.

For each positive integer $j$,  we show that  $(y,z)B  \cap B_j = (y,z)B_j$.
It is clear that $(y,z)B_j \subseteq (y,z)B \cap B_j$.   To show the reverse
inclusion, it suffices to show
for each integer $j \ge 0$ that $(y,z)B_{j+1} \cap B_j \subseteq (y,z)B_j$.
We  have $B_j[f_{j+1}] \subseteq (B_j)_{(y,z)B_j}$
since $f_{j+1} = \frac{f_j}{x}~ + ~ yc_{j+1} + zd_{j+1}$  by (\ref{16.5.4de}.1).
  The center of the 2-dimensional RLR $(B_j)_{(y,z)B_j}$     on $B_j[f_{j+1}]$ is
the prime ideal $(y,z)B_j[f_{j+1}] $.  This prime ideal is contained
in the maximal ideal $(x,y,z,f_{j+1})B_j[f_{j+1}]$;   it follows that
$B_{j+1} \subseteq (B_j)_{(y,z)B_j}$ and so $(y,z)B_{j+1}\cap B_j\subseteq (y,z)B_j$.

Thus  $(y,z)B \cap B_{n+1}  = (y,z)B_{n+1}$, and $f_{n+1} \in (y,z)B_{n+1}$.
Since $x,y,z$ and $f_{n+1}$ are algebraically independent variables over $k$, and $B_{n+1}=k[x,y,z,f_{n+1}]_{(x,y,z,f_{n+1})}$,   this is a contradiction.
 We conclude
that $Q$ is not finitely generated.

By item~3 the ring $B_Q$ is a three-dimensional regular local ring. Since $x$ is a unit of $B_Q$ and since
$Q = (y,z,f, f_1, f_2, \dots)B$, it follows from Proposition~\ref{16.5.15}.2  ($a=x$ and $C=B$) that
$QB_Q  = (y,z,f)B_Q$. This establishes item~5.

For item 6, since $x \not\in Q$ and
$B/xB \cong R/xR$, there are infinitely many height-two primes of $B$
containing $xB$. This proves there are infinitely many height-two primes of $B$ not
contained in $Q$.   If $P$ is a height-two prime of $B$ not contained in $Q$, then $\hgt(\m_B/P) = 1$, by item 4 above,
and so, by Proposition~\ref{16.5.15}.5, $P$ is contracted from $R^*$. This completes item~6.

For   item~7 we  show that $p = zB$
has the stated properties. By Corollary~\ref{homimB-z}, the ring $B/zB$ is
isomorphic to the ring called $B$ in Example~\ref{16.5.3de}. For convenience we relabel the
ring of Example~\ref{16.5.3de} as $B'$. By Theorem~\ref{16.5.2},
$B'$  has exactly one non-finitely generated prime ideal, which we label $Q'$, and $\hgt Q'=2$.  It follows that
$Q/zB= Q'$.    By Discussion~\ref{types}, there are infinitely many height-one primes contained in $Q'$
of Type    II (that is, primes that are contracted from $R^*/zR^*$) and infinitely many
height-one primes
 contained in $Q'$
of Type    III (that is, primes that are not contracted from $R^*/zR^*$).
The preimages in $R^*$ of these  primes are  height-two primes
of $B$ that are contained in $Q$ and contain $zB$. It follows that
there are infinitely many contracted from $R^*$ and there are
infinitely many not contracted from $R^*$, as desired for item~7.

For item~8, we have a saturated chain of prime ideals
$$
(0)~ \subset ~ xB ~ \subset ~(x,y)B ~ \subset ~ (x,y,z)B ~= ~ \m_B
$$  of length 3 by Equation~\ref{16.5.15}.0  We have a saturated chain of prime ideals
$$
(0)~  \subset ~  yB ~ \subset~ (y,z)B ~ \subset ~ Q ~\subset ~ \m_B
$$ of length 4 from the proof of item~4.  Hence $B$ is not catenary.
By item~2, $\dim B = 4$, so there is no saturated chain of prime ideals of $B$ of length greater
than 4. By Comments~\ref{16.5.4t}.0, no saturated chain of prime ideals of $B$ has length
less than 3.

For item~9, since $R^*$ is a Krull domain  and $B = A = \mathcal Q(B) \cap R^*$, it
follows that $B$ is
a Krull domain and  each height-one prime of $B$ is the contraction of a
height-one prime of $R^*$.
Since B/xB and B[1/x] are Noetherian, item~10 follows from \cite[Corollary~1.3]{HR}.
\end{proof}

\begin{remarks} \label{16.5.41r} Let the  notation  be as in Theorem~\ref{16.5.4t}.

 (1)  It follows from Theorem~\ref{16.5.4t}
that the localization
$B[1/x]$ has a unique maximal ideal $QB[1/x] = (y,z,f)B[1/x]$ of height
three and has infinitely many maximal ideals of height two. We observe that  $B[1/x]$ has no
maximal ideal of height one. To show this last statement it suffices to show for each irreducible
element $p$ of $B$ with $pB \ne xB$  there exists $P \in \Spec B$ with $pB \subsetneq P$
and $x \not\in P$. Assume there does not exist such a prime ideal $P$. Consider the ideal
$(p, x)B$. This ideal has height two and has only finitely many minimal primes since $B/xB$
is Noetherian. Let $g$ be an element of $\m_B$ not contained in any of the minimal primes of
$(p,x)B$. Every prime ideal of $B$ that contains $(g,p)B$ also contains $x$ and hence has height
greater than two. Since $x \notin Q$, it follows that $(g,p)B$ is $\m_B$-primary,
 and hence that $(g,p)R^*$ is $\m_{R^*}$-primary. Since $R^*$ is Noetherian and  $\hgt \m_{R^*} = 3$,
 this contradicts the  Altitude Theorem
of Krull \cite[Theorem~9.3]{N2}.

(2)  Every ideal $I$ of $B$ such that $IR^*$ is $\m_{R^*}$-primary is $\m_B$-primary
by Proposition~\ref{16.5.15}.5.

(3)     Define
$$
C_n  ~:=~ \frac{B_n}{(y,z)B_n} \quad \text{ and  }  \quad C ~:= ~ \frac{B}{(y,z)B}.
$$
We have $C = \bigcup_{n=0}^\infty C_n$  by item 1.   We show that $C$ is a rank 2
valuation domain with principal maximal ideal generated by the image of $x$.
For each positive integer $n$,
let $g_n \in C_n$ denote the image of the element  $f_n$ and let $x$ denote the image of $x$.
Then $C_n = k[x, g_n]_{(x,g_n)}$ is a
2-dimensional RLR.
By (\ref{16.5.4de}.1), $f_n = xf_{n+1} + x(c_{n}y + d_{n}z)$.  It follows that
 $g_n = xg_{n+1}$ for each $n \in \N$.  Thus $C$ is an infinite directed union
of quadratric transformations of 2-dimensional regular local rings.
Thus $C$ is a valuation domain  of dimension at most 2  by \cite{Abhy}.  By
items~2 and 4 of Theorem~\ref{16.5.4t}, $\dim C \ge 2$, and therefore $C$ is a valuation
domain of rank 2.
 The maximal ideal of $C$
is $xC$.

By Corollary~\ref{homimB-z},  $B/zB\cong D$, where $D$ is the ring $B$ of
Example~\ref{16.5.3de}. By an argument similar to that of  Proposition~\ref{homimB} and
Corollary~\ref{homimB-z}, we  see that the above ring $C$ is isomorphic to $D/yD$.
\end{remarks}

\begin{question} \label{17.5.12}  For the ring  $B$
constructed as in Example \ref{16.5.4de}, we ask:
 Is $Q$ the only prime ideal of $B$ that is not finitely generated?
\end{question}

Theorem~\ref{16.5.4t} implies that the only possible nonfinitely generated prime
ideals of $B$ other than  $Q$ have height two. We do not know whether every height-two
prime ideal of $B$ is finitely generated. We show  in Corollary~\ref{litNoeth} and
Theorem~\ref{litnonN}
that certain of the height-two primes of $B$
are finitely generated.

\smallskip

Lemma~\ref{flat} is the key to the proof of Theorem~\ref{build} and is also
useful below. We are grateful to Roger Wiegand for observing it.

\begin{lemma}  $\phantom{x}$\cite[Lemma~3.1]{noehom}, \cite[Lemma~8.2]{powerbook}  \label{flat}
Let $S$ be a subring of a ring $T$ and
let $b \in S$ be a regular element of both $S$ and $T$.
Assume that $bS = bT \cap S$ and $S/bS = T/bT$. Then
\begin{enumerate}
\item $T[1/b]$ is flat over $S$ $\iff$ $T$ is flat over $S$.
\item If $T$ and $S[1/b]$ are both Noetherian and $T$ is
flat over $S$, then $S$ is Noetherian.
\end{enumerate}
\end{lemma}

The following theorem shows that the nonflat locus  of the
map $\varphi: B \to R^*[1/a]$  yields flatness for certain homomorphic images of $B$, if $R, a, R^*$ and $B$ are as in the general construction outlined in Theorem~\ref{build}.

\begin{theorem}\label{Noeth} Let $R$ be a Noetherian integral domain with field of
fractions $K$, let $a \in R$ be a nonzero nonunit,
and let
$R^*$ denote  the $(a)$-adic
completion of $R$. Let $s$ be a positive integer and let ${\underline\tau}=\{\tau_1,\dots,\tau_s\}$ be 
a set of elements of $R^*$ that are algebraically independent over $K$, so 
that $R[\underline\tau]$ is  a polynomial ring in $s$ 
variables over $R$. Define $A := K(\underline\tau)  \cap  R^*$. Let  $U_n, B_n$, $B$ and $U$ be defined as 
follows
$$
U~:= ~\bigcup_{r=0}^\infty U_{n}  \quad \text{and}
\quad B ~:=~
\bigcup_{n=0}^\infty B_{n},
$$
where, for each integer $n \ge 0$, $U_{n} := R[\tau_{1n},
\ldots, \tau_{sn}]$, ~
$B_{n} := (1 + aU_{n})^{-1}U_{n}$, and each
$\tau_{in}$ is the $n^{\text{th}}$ endpiece of $\tau_i$ defined as in Equation~\ref{16.5.1}.a. 
Assume that  $F$ is  an ideal of $R^*[1/a]$ that defines the nonflat locus of the
map $\varphi: B \to R^*[1/a]$.
Let $I$ be
an ideal in $B$ such that $IR^* \cap B = I$ and $a$ is regular on $R^*/IR^*$.

\begin{enumerate}
\item  If   $IR^*[1/a] + F = R^*[1/a]$, then   $\varphi \otimes_B (B/I)$ is flat.
\item
If $R^*[1/a]/IR^*[1/a]$ is flat over $B/I$, then $R^*/IR^*$ is flat over $B/I$.
\item
If $\varphi \otimes_B (B/I)$ is flat, then $B/I$ is Noetherian.
\end{enumerate}
\end{theorem}

\begin{proof} The hypothesis of item~1 implies that $\varphi_P$ is flat
for each $P \in \Spec R^*[1/a]$ with $I \subseteq P$. Hence for each
such $P$ we have $\varphi_P \otimes_B (B/I)$ is flat. Since flatness is
a local property, it follows that $\varphi \otimes_B (B/I)$ is flat.

For  items~2 and 3,  apply Lemma~\ref{flat}  with $S=B/I$ and $T=R^*/IR^*$;
the element $b$ of Lemma~\ref{flat} is the image in $B/IB$ of the element $a$
from the setting of Theorem~\ref{build}.
Since $IR^* \cap B = I$, the ring $B/I$ embeds into $R^*/IR^*$, and since $B/aB = R^*/aR^*$, the
ideal $a(R^*/IR^*) \cap (B/I) = a(B/I)$.
Thus  item~2 and item~3 of Theorem~\ref{Noeth} follow
from item~1 and item~2, respectively, of Lemma~\ref{flat}.
\end{proof}

\begin{corollary}\label{litNoeth} Assume the notation of Example~\ref{16.5.4de}. Let $w$ be a prime
element of $B$.   Then $B/wB$ is Noetherian if and only if $w \notin Q$.  Thus 
every nonfinitely generated
ideal of $B$ is contained in $Q$.
\end{corollary}
\begin{proof} If $w \in Q$, then $B/wB$  is not Noetherian since $Q$ is not finitely
generated. Assume $w \notin Q$. Since $B/xB$ is known to be Noetherian, we may
assume that $wB \ne xB$. By Proposition~\ref{nfl}.1,  $QR^*[1/x] = (y,z)R^*[1/x]$
defines the nonflat locus of
$\varphi: B \to R^*[1/x]$. Since $wR^*[1/x]+(y,z)R^*[1/x]=R^*[1/x]$,
Theorem~\ref{Noeth} with $I = wB$  and $a = x$  implies that $B/wB$ is Noetherian.

For the second statement, we use that every nonfinitely generated ideal is contained in
an ideal maximal with respect to not being finitely generated and the latter ideal is prime.
Thus
it suffices to show every prime ideal $P$ not contained in $Q$ is finitely generated.
If $P \not\subseteq Q$, then, since $B$ is a UFD,  there exists a prime element
$w \in P \setminus Q$. By the first statement, $B/wB$ is Noetherian, and so $P$ is
finitely generated.
\end{proof}

\begin{theorem}\label{litnonN} Assume the notation of Example~\ref{16.5.4de}.
Let $w$ be a prime element of $R$ with $w \in (y,z)k[x,y,z]$. If $w$ is linear in
either $y$ or $z$, then $Q/wB$ is the unique nonfinitely generated prime
ideal of $B/wB$. Thus $Q$ is the unique nonfinitely generated prime ideal of $B$ that
contains $w$.
\end{theorem}

\begin{proof}
Let $\overline{\phantom{x}}$ denote image under the canonical map $\pi:R^*\to R^*/wR^*$. We may assume that $w$ is linear in $z$, that the coefficient of $z$ is $1$ and therefore that
$w=z-yg(x,y),$
where  $g(x,y)\in k[x,y]$.  Thus
$\overline R\cong k[x,y]_{(x,y)}$.
By Proposition~\ref{homimB} $\overline B$ is the approximation domain over $\overline R$ with respect to the transcendental  element $$\overline f=\overline y\cdot\overline\tau+\overline z\cdot\overline\sigma=\overline y\cdot\overline\tau+\overline y\cdot\overline{g(x,y)}\cdot\overline\sigma .$$
The setting of
Proposition~\ref{Bufd} applies with $C = \overline{B}$, the underlying ring $R$ replaced by $\overline R$,   and $a = \overline{x}$. Thus the ring $\overline B$ is a UFD, and so
 every height-one prime ideal of $\overline B$ is principal. Since $w \in Q$ and $Q$ is not finitely
generated, it follows that $\hgt(\overline Q) = 2$ and that
$\overline Q$ is
the unique nonfinitely generated prime ideal of $\overline B$. Hence the theorem holds.
 \end{proof}

\begin{remarks} \label{16.5.4tr}
It follows from
Proposition~\ref{16.5.15}.5 that every height two prime of $B$ that is not contained
in $Q$ is contracted from a prime ideal of $R^*$.   As  we state in item 7 of Theorem~\ref{16.5.4t},
there are infinitely many height-two prime ideals of $B$ that are contained in $Q$ and are
contracted from $R^*$ and  there are infinitely many height-two prime ideals of $B$ that
are contained in $Q$ and are {\it not} contracted from $R^*$.  In particular infinitely many of each type exist between $zB$ and $Q$, and similarly also infinitely many of each type exist between $yB$ and $Q$.

 Since $B_Q$ is a 3-dimensional regular local ring, for each
height-one prime $p$ of $B$ with $p \subset Q$, the set
$$
\mathcal S_p ~= ~ \{P \in \Spec B ~| ~ p \subset P \subset Q \text{ and } \hgt P = 2 \}
$$
is infinite. The infinite set $\mathcal S_p$ is the disjoint union of the sets $\mathcal S_{pc}$ and
$\mathcal S_{pn}$, where the elements of $\mathcal S_{pc}$ are contracted from $R^*$ and
the elements of $\mathcal S_{pn}$ are not contracted from $R^*$.

 We do not know  whether there exists a height-one prime $p$  contained in $Q$  having
the property that
one of the sets $\mathcal S_{pc}$ or $\mathcal S_{pn}$ is empty. Furthermore if one of these sets is empty,
which one is empty? If there are some such height-one primes $p$ with one of the sets $\mathcal S_{pc}$ or $\mathcal S_{pn}$ empty, which height-one primes are they? It would be interesting
to know the  answers to these questions.
\end{remarks}

The referee of this article asked how Example~\ref{16.5.4de} compares to a specific ring 
constructed using the popular ``$D + M$'' technique of multiplicative ideal theory;  see for
example \cite[p. 95]{Gil}, \cite{Gilb} or \cite{BR}. The ``$D + M$'' construction involves an
integral domain $D$ and a prime ideal $M$ of an extension domain $E$ of $D$ such that 
$D \cap M = (0)$. Then  
$ D + M = \{a+b ~|~ a \in D, ~  b \in M \}$.  Moreover, for $a, a' \in D$ and 
$b, b' \in M$, if $a + b = a' + b'$, then $a = a'$ and $b = b'$. 
Since $D$ embeds in $E/M$,  the
ring $D + M$ may be regarded as a pullback as in \cite{GH} or \cite[p. 42]{LW}. 
The  ring suggested by the referee  is an interesting example that contrasts nicely with 
Example~\ref{16.5.4de}.  We describe it in Example~\ref{referee}.

\begin{example} \label{referee}
Let $B$ be the ring of Example~\ref{16.5.3de}. Then $B = k + \m_B$ in the notation 
of Example~\ref{16.5.3de}. Assume the field $k$ is the field of fractions of a DVR $V$,
and let $t$ be a generator of the maximal ideal of $V$. Define 
$$ 
C ~ := ~  V ~ + ~  \m_B ~= ~ \{~ a ~+~ b ~~|~ ~ a ~\in ~ V,~ ~ b ~\in~ \m_B ~\}.
$$

The integral domain $C$ has the following properties:
\begin{enumerate}
\item The maximal ideal $\m_B$ of $B$ is also a prime ideal of $C$,  and 
$C/\m_B \cong V$.
\item $C$ has a unique maximal ideal $\m_C$;  moreover,  $\m_C  = tC$. 
\item $\m_B = \bigcap_{n=1}^\infty t^nC$,  and  $B = C_{\m_B} = C[1/t]$.
\item Each $P \in \Spec C$ with $P \ne \m_C$
is contained in $\m_B$; thus $P \in \Spec B$. 
\item $\dim C = 4$ and 
 $C$ has a unique prime ideal of height $h$, for $h = 2, 3$ or $4$.
\item  $\m_C$ is the only nonzero  prime ideal of $C$ that is finitely generated.
Indeed, every nonzero proper ideal of $B$ is an ideal of $C$ that is not finitely 
generated.
\end{enumerate}
Thus $C$ is a non-Noetherian non-catenary four-dimensional local domain.  
\end{example}
\begin{proof}
Since $C$ is a subring of $B$,  
$\m_B \cap V = (0)$ and  $V\m_B = \m_B$,   item 1 holds. 
We have  $C/(tV + \m_B) = V/tV$.  Thus  
$tV + \m_B$ is a maximal ideal of $C$. 
Let  $b \in \m_B$.   
Since $1 + b$ is a unit of 
the local ring $B$,  we have
$$
\frac{1}{1+b}  ~=~ 1 - \frac{b}{1+b} \quad \text{ and } \quad \frac{b}{1+b} ~\in ~\m_B. 
$$
Hence  $1 + b$ is a unit of $C$.
Let 
$a + b \in C \setminus (tV + \m_B)$, where $a \in V \setminus tV$ and $b \in \m_B$. 
Then $a$ is a unit of $V$ and thus a unit of $C$. Moreover,  $a^{-1}(a+b) = 1 + a^{-1}b$
and $a^{-1}b \in \m_B$. Therefore $a + b$ is a unit of $C$. We 
conclude that $\m_C := tV + \m_B$ is the unique maximal ideal of $C$. 
Since $t$ is a unit of $B$,
we have $\m_B = t\m_B$. Hence $\m_C = tV + \m_B = tC$. 
This proves item~2.  

For item~3, since $t$ is a unit of $B$, we have  $\m_B = t^n\m_B \subseteq t^nC$
for all $n \in \N$. Thus $\m_B \subseteq \bigcap_{n=1}^\infty t^nC$. If 
$a + b \in \bigcap_{n=1}^\infty t^nC $  with $a \in V$ and $b \in \m_B$, 
then 
$$ 
b ~\in ~ \bigcap_{n=1}^\infty t^nC ~\implies ~ a ~\in ~(\bigcap_{n=1}^\infty t^nC) \cap V ~
= ~ \bigcap_{n=1}^\infty t^nV ~= ~ (0).
$$
Hence $\m_B = \bigcap_{n=1}^\infty t^nC$.  Again using  $t\m_B = \m_B$, we obtain  
$$
C[1/t]~ = ~ V[1/t] ~ + ~ \m_B ~ =  ~k ~ + ~\m_B ~ = ~B.
$$
Since $t \notin \m_B$, we have $B = C[1/t] \subseteq C_{\m_B} \subseteq B_{\m_B} = B$. This 
proves item~3.

For $P$ as in  item~4, we have $P \subsetneq tC$. Since $P$ is a prime ideal of $C$, it 
follows that $P = t^nP$ for each $n \in \N$. By item~3, $P \subseteq \m_B$, and it follows
that $P \in \Spec B$. Item~5 now follows from item~4 and the structure of $\Spec B$.

For item~6, let $J$ be a nonzero proper ideal of $B$. Since $t$ is a unit of $B$,
we have $J = tJ$. This  implies by Nakayama's Lemma that $J$ as an ideal of $C$ 
is not finitely generated; see \cite[Lemma~1]{BR}. Thus item~6 follows from item~4. 

By item~6,  $C$ is non-Noetherian. Since $(0) \subsetneq xB \subsetneq \m_B \subsetneq tC$ 
is a saturated chain of prime ideals of $C$ of length 3, and 
$(0) \subsetneq yB \subsetneq Q \subsetneq \m_B \subsetneq tC$ 
is a saturated chain of prime ideals of $C$ of length 4, the ring $C$ is not catenary.
\end{proof}

\end{document}